\documentclass[DIV=10]{scrartcl}
\usepackage{graphicx} 
\usepackage{hyperref}
\usepackage{titling}
\thanksmarkseries{arabic}
\usepackage{amsfonts}
\usepackage{amsmath}
\usepackage{amssymb}
\usepackage{amsthm}
\usepackage[capitalize, nameinlink]{cleveref}
\addtokomafont{disposition}{\rmfamily}
\usepackage{microtype}
\usepackage{xcolor}
\usepackage{tikz}
\usepackage{multirow}
\usepackage{hyperref}

\makeatletter
\let\blx@rerun@biber\relax
\makeatother

\newcommand{\R}{\mathbb{R}}

\newcommand{\E}{\mathbb{E}}
\newcommand{\N}{\mathbb{N}}

\newtheorem{assumption}{Assumption}
\newtheorem{lemma}{Lemma}
\newtheorem{remark}{Remark}
\newtheorem{theorem}{Theorem}
\newtheorem{corollary}{Corollary}

\newtheorem{definition}{Definition}

\title{Learned Regularization for Inverse
Problems: Insights from a Spectral Model}
\author{Martin Burger\thanks{Helmholtz Imaging, Deutsches Elektronen-Synchrotron DESY, Notkestr. 85, Hamburg, 22607, Germany} \textsuperscript{,}\thanks{Fachbereich Mathematik, Universit\"{a}t Hamburg, Bundesstrasse 55, Hamburg, 20146, Germany}
\and 
Samira Kabri\textsuperscript{$1$}}
\date{Corresponding author:
\href{mailto:samira.kabri@desy.de}{samira.kabri@desy.de}}

\begin{document}

\maketitle
\abstract{\noindent In this chapter we provide a theoretically founded investigation of state-of-the-art learning approaches for inverse problems from the point of view of spectral reconstruction operators. We give an extended definition of regularization methods and their convergence in terms of the underlying data distributions, which paves the way for future theoretical studies.
  
  Based on a simple spectral learning model previously introduced for supervised learning, we investigate some key properties of different learning paradigms for inverse problems, which can be formulated independently of specific architectures. In particular we investigate the regularization properties, bias, and critical dependence on training data distributions. Moreover, our framework allows to highlight and compare the specific behavior of the different paradigms in the infinite-dimensional limit.}
\section{Introduction} 
With the rise of deep learning and generally data-driven methods in all fields of research,  numerous promising advances that employ data-driven strategies for the solution of inverse problems have emerged in recent years.
End-to-end approaches, which aim to learn the entire map from measurements or naive reconstructions to desirable reconstructions (see e.g., \cite{FBPConvNet}) often outperform theoretically founded reconstruction techniques, which do not take into account specific information on the data under consideration. Interestingly, end-to-end approaches still seem to be successful for limited access to the data, as the Data Challenges organized annually by the Finnish Inverse Problems Society\footnote{\url{https://www.fips.fi/challenges.php}} reveal: All teams that won the first place in the last three years enlarged the small set of provided training data by creating synthetic data to then train a neural network either on basic reconstructions or even directly on the measurements (see \cite{Trippe2021} for the 2021 Helsinki Deblur Challenge,  \cite{Germer2022} for the 2022 Helsinki Tomography Challenge and \cite{Denker2023} for the 2023 Kuopio Tomography Challenge). To improve the interpretability and robustness of data-driven reconstructions, many approaches aim to integrate the expressiveness of neural networks into the theoretical framework provided by the classical theory. The concept of Tikhonov regularization is used {e.g.,} in \cite{chung2011designing, bauermeister2020learning} to estimate optimal spectral filters for a given data-set and \cite{alberti21} studies the data-driven optimization of linear but not necessarily diagonal Tikhonov regularizers in an infinite-dimensional setting. The NETT (Network Tikhonov) approach \cite{Li_2020} studies the possibilities to include more general learned regularization functionals into variational optimization problems, which often do not have a closed-form solution and thus have to be solved iteratively. Incorporating learning directly into an iterative optimization scheme by so-called unrolling was successfully used for example in \cite{LearnedPrimalDual, FISTANet}.   Another popular direction is the usage of plug-and-play priors (cf. \cite{venkatakrishnan2013}) or, to put it in more general terms, regularization by denoising (see e.g., \cite{meinhardt2017learning,romano2017little}). Similar to the unrolling approach, the basic idea of plug-and-play priors is to substitute the proximal operator in an iterative scheme by a parametrized map. The main difference, however, is that this map is not optimized in an end-to-end approach that requires explicit knowledge of the forward operator. Instead, it is trained as a denoiser on the desired output space and therefore promises high flexibility between various applications. Another approach in the spirit of plug-and-play would be to apply an already trained denoiser on naive reconstructions as a post-processor. 
\\An important aspect of the reconstruction of unknown quantities from noisy measurements is the stability of the reconstruction operator, meaning that the reconstruction errors can be controlled by the measurement errors in some way. Since naive reconstruction operators obtained by a formal inversion of the measurement operator are often unstable, the concept of regularization (cf. \cite{benning2018modern, engl1996regularization}) is inevitable. Here, the reconstruction operator is not only supposed to be stable, but also to approximate a (generalized) inverse of the forward operator, as the noise corrupting the data vanishes. Despite their importance, these properties have not been well analyzed in the data-driven context so far. Usually, data-driven approaches are derived in a finite dimensional setting, since the data is only accessible in a discretized form. Thus, in many cases, for example in approaches using neural networks with Lipschitz continuous activation functions, the obtained reconstruction operators are Lipschitz continuous by construction. Therefore, to precisely assess their stability, one would have to compute their Lipschitz constant, which is, in general, NP-hard (cf. \cite{virmaux2018lipschitz}). {Even worse we would need to compute the dependence of the Lipschitz constant on the data.}  The infinite-dimensional limit and corresponding regularization properties in function spaces are even less understood. 
\\Furthermore, as most approaches are designed to generate solution operators for problems with a fixed noise model, there are only a few studies on the convergence behavior of data-driven methods in the no-noise limit:
 The convergence of the already mentioned NETT approach is discussed in \cite{Li_2020}. Deploying invertible residual networks (iResNets) to approximate the forward operator well enough to obtain a convergent regularization method is analyzed in \cite{arndt2023invertible} and investigated from a Bayesian perspective in \cite{arndt2023bayesian}. 
The convergence behavior of plug-and-play priors with contractive denoisers is examined in \cite{ebner2022plugandplay}, while \cite{hauptmann2023convergent} shows that plug-and-play reconstruction with linear denoisers can be obtained by spectral filtering of the denoiser. The survey \cite{mukherjee2023learned}  recapitulates various convergence concepts and connects heuristic approaches from practice to mathematical theory.
Although the statistical nature of the noise is taken into account during the training process in most of these works, the convergence itself is shown under the assumption of deterministically bounded noise. Moreover, to the best of our knowledge, the existing convergence studies assume the possibility to manually tune an additional regularization parameter, which is not obtained during a training procedure (we refer to \cite{rodriguez2023learning} for a study on the data-driven optimization of a parameter by empirical risk minimization). 
\\In this chapter, we aim to adapt this classical view of a convergent regularization method to the practical considerations of data-driven strategies. In both settings, a regularization is a family of continuous, parametrized reconstruction operators. However, the idea of a data-driven approach is to determine these parameters by minimizing a risk functional that can depend on the distribution of the  measurements and (for supervised approaches) also on the distribution of the ground truth data. Given that the risk functional and the method to obtain its minimizer are fixed, the user does not have to provide the parameters, but rather the distributions that determine the risk functional, usually in the form of samples. We call these distributions the training data, and, if the noise and measurements can be modeled separately, the training noise. The approach is then called a data-driven regularization, if for any noise model in the real problem there exists a choice of training data and/or training noise such that the reconstruction error converges to zero as the problem noise vanishes. In consideration of plug-and-play approaches, we are particularly interested in ways to choose the training noise only depending on the noise level, such that learned regularizers can be transferred to a broad class of problems. \\
In order to gain further understanding, we want to develop a framework that is applicable to infinite dimensional settings and further includes the case of infinite dimensional white noise in the form of a random process (cf. \cite{FranklinNoise}). Therefore, we propose a new definition of the noise level that is coherent with the considerations for Gaussian white noise made in \cite{KekkonenNoise}. The model we use to understand different training approaches is a spectral regularization framework reminiscent of classical linear regularization methods (cf. \cite{engl1996regularization}) and further developed as a model for supervised learning in \cite{bauermeister2020learning,kabri2024convergent}. With this simple architecture allowing for rather explicit computations, we first revisit supervised learning in our setting. To reduce the dependence of the forward operator during training, we further transfer our results to the plug-and-play setting and derive conditions for convergence under restricted knowledge of the noise model. Our study is completed by the investigation of so-called adversarial regularizers (cf. \cite{lunz18,mukherjee2021learning}), which aim to train the regularizer as a discriminator between the distributions of ground truth and non-regularized data.\\
The remainder of this chapter is organized as follows: 
Based on the framework developed in \cref{sec:theory}, we study the regularization properties of linear spectral reconstruction operators obtained by various data-driven reconstruction strategies from three different paradigms, namely supervised learning (\cref{sec:mse}), regularization by denoising (\cref{sec:denoising}) and adversarial regularization (\cref{sec:adversarial}). We compare and summarize our findings in \cref{sec:comp}.
\section{Convergent Data-Driven Regularization} \label{sec:theory}

In the following, we discuss the basics of convergent data-driven regularization. We start with the basic setup of linear operators we shall use in our study, and then present also the statistical model of the noise. In the last part, we {provide} our novel definitions of data-driven regularization methods, which actually go beyond the specific setup here and even extend to nonlinear inverse problems. 

\subsection{Reconstruction with linear operators and spectral decomposition}
For our studies, we consider a compact linear forward operator $A: X \rightarrow Y$ between  (infinite-dimensional) Hilbert spaces $X$ and $Y$.
For the sake of simplicity, we additionally assume that $A$ has an infinite-dimensional range, such that the {forward problem is ill-posed and} regularization is really needed from a functional-analytic perspective.  Measurements $y \in Y$ that are generated by applying $A$ to ground truth data $x \in X$ are assumed to be corrupted by noise $\epsilon$, leading to 
\begin{align*}
    y = Ax + \epsilon.
\end{align*}
To make explicit computations possible, we resort to linear spectral regularizers as a simple but yet expressive architecture of the reconstruction methods, that have been used in the data-driven setting for example in \cite{arndt2023invertible,arndt2023bayesian,bauermeister2020learning, kabri2024convergent,chung2011designing}. The compactness of $A$ allows for its representation by a discrete singular system, more precisely by a set $\{\sigma_n;\; u_n, v_n\}_{n \in \N}$ where $\{u_n\}_{n\in \N}$ is an orthonormal basis of $N(A)^\perp$, the orthogonal complement of the null-space of $A$ and $\{v_n\}_{n\in \N}$ is an orthonormal basis of $\overline{R(A)}$, the closure of the range of $A$. For each $n$, $\sigma_n > 0$ is the singular value related to $u_n$ and $v_n$ as $Au_n = \sigma_n v_n$. The operator $A$ and its pseudo-inverse $A^{\dagger}$ can thus be expressed as 
\begin{align*}
    Ax = \sum_{n\in \N} \sigma_n \langle x, u_n \rangle v_n \qquad \text{and} \qquad A^{\dagger}y = \sum_{n\in \N} \frac{1}{\sigma_n} \langle y, v_n \rangle u_n.
\end{align*}
In a linear spectral regularizer, the factors $\frac{1}{\sigma_n}$ are substituted by suitable filtering coefficients $g = \{g_n\}_{n\in\N}$, which leads to the reconstruction operator
\begin{align}\label{eq:linspecreg}
    R(y; g) = \sum_{n \in \N} g_n \langle y,v_n\rangle u_n.
\end{align} 
Since this architecture is inherently connected to the spectral decomposition of $A$ into the singular system, it will be called the \textit{spectral architecture} in the following. 
As shown in \cite{ebner2023regularization} and \cite{Hubmer_2022} a generalization of the regularization based on the singular value decomposition can be achieved by using a diagonal frame decomposition.  
For the sake of simplicity, we assume the multiplicity of each singular value to be one and refer to \cite[Remark 2]{kabri2024convergent} for the more general case.
In the following, we analyze different data-driven approaches to determine the coefficients $g_n$ and assume data drawn from a distribution $\pi$ and uncorrelated noise drawn from a distribution $\mu$. The approaches we want to analyze in this chapter yield coefficients $g_n$ that can be described by the singular values $\sigma_n$ and the terms
\begin{align}\label{eq:pidelta}
\Pi_n(\pi) = \E_{x\sim \pi}\left[\langle x, u_n \rangle^2\right] \qquad \text{and} \qquad \Delta_n(\mu) = \E_{\epsilon \sim \mu}\left[\langle \epsilon, v_n \rangle^2\right].
\end{align}   
Roughly speaking, the terms in \eqref{eq:pidelta} describe the variance of data and noise in the direction of the singular functions. 
We will further see that assuming uncorrelated noise yields a Tikhonov-like reconstruction. Such a reconstruction could be equivalently obtained in a variational approach of the form 
\begin{equation}\begin{aligned}\label{eq:tikhRegu}
\operatorname*{arg\,min}_{x' \in X} \frac{1}{2} \|Ax' - y\|^2 +& \frac{1}{2} J_{\lambda}(x')=  \sum_{n \in \N} \frac{\sigma_n}{\sigma_n^2 + \lambda_n} \langle y, v_n \rangle u_n,
\end{aligned}\end{equation}
where the regularization functional 
\begin{align}\label{eq:Jlambda}
    J_{\lambda}(x) = \sum_{n \in \N} \lambda_n \langle x,u_n \rangle^2,
\end{align}
is parametrized by $\lambda = \{\lambda_n\}_{n\in\N}$ with $\lambda_n \geq 0$.

 We mainly focus on the continuity and the convergence behavior of the data-driven approaches as the underlying noise distribution changes.
Therefore, we assume the forward operator, and thus its singular system, to be fixed throughout this chapter. For the distribution of the training data, $\pi$, we make the following assumptions: 
\begin{assumption}\label{ass:data}
    \item The covariance operator of the data distribution  is assumed to be a trace-class operator, which means in particular that 
        \begin{align*} \sum_{n\in \N} \Pi_n(\pi ) = \sum_{n\in \N} \E_{x \sim \pi }\left[ \langle x, u_n \rangle^2\right] = \E_{x \sim \pi }\left [\sum_{n\in \N} \langle x, u_n \rangle^2\right] = \E_{x \sim \pi }(\| x \|^2) < \infty.\end{align*} 
Additionally, we assume
$\Pi_n(\pi ) > 0$ 
for all $n \in \N$.
\end{assumption}
By choosing the specific architecture \eqref{eq:tikhRegu} we limit ourselves to approximations of the pseudo-inverse $A^\dagger$, which achieves unique solutions by neglecting the null-space components of possible reconstructions. Therefore, the expected reconstruction error for uncorrupted measurements is bounded from below by the expected data variance in the null-space of the operator, as
\begin{align}\label{eq:bias_with0}
    \mathbb{E}_{x\sim \pi }\left [ \| x - R(Ax;\,g)\|^2 \right]  =\mathbb{E}_{x \sim \pi } \left [ \|x_0\|^2\right] + 
    \sum_{n\in \N} (1-\sigma_n g_n)^2 \Pi_n(\pi ),
\end{align}
where $x_0$ denotes the projection of $x$ on the null-space $N(A)$. The recovery of a non-trivial null-space component can still be of great importance, for example in medical imaging with limited data. Examples include limited angle tomography (see for example \cite{bubba2021deep, Liu_2023_ICCV}) or magnetic resonance imaging (see for example \cite{hyun2018deep}). For a study on more general data-driven Tikhonov regularizers in the infinite dimensional setting, which allow for a bias and are not necessarily diagonal with respect to the singular system of the forward operator, we refer to \cite{alberti21}. A general framework to approximate null-space components with deep neural networks in a second step after applying a conventional reconstruction method is introduced in  \cite{schwab2019deep}.\\
Although we do not take into account the quantity $\mathbb{E}_{x \sim \pi } \left [ \|x_0\|^2\right]$ in this chapter, we are still interested in the remaining part of the error. We therefore define the bias of the reconstruction approach $R(\cdot;\,g)$ based on the data distribution $\pi $ as
\begin{align*}
e_0 =
    \sum_{n\in \N} (1-\sigma_n g_n)^2 \Pi_n(\pi )= \sum_{n\in \N} \left ( \frac{\lambda_n}{ \sigma_n^2 + \lambda_n}\right)^2 \, \Pi_n(\pi ).
\end{align*}
While we cannot expect the bias of a continuous reconstruction approach to be zero, it is still desirable that it converges to zero for a vanishing corruption of the measurements.  
\subsection{Statistical noise level}
Apart from being stable, a regularization method has to converge to a generalized inverse as the noise level tends to zero. In case of deterministic noise $\epsilon$, this noise level is often chosen as the norm of the noise, i.e., $\delta = \|\epsilon\|$.
This approach can be generalized to the statistical setting (cf. \cite{benning2018modern}) by defining the noise level corresponding to a random variable $\epsilon$ drawn from a distribution $\mu$ as $\delta^2 = \E_{\epsilon\sim \mu} \left[\|\epsilon\|^2\right]$.
Of course, this definition only makes sense if $\|\epsilon\| < \infty$ almost surely, which is not the case for $\|\cdot\|$ being the $L^2$-norm and Gaussian white noise in an infinite dimensional setting, as it is pointed out for example in \cite{KekkonenNoise}. There, Gaussian white noise of level $\delta$ is modelled by $\delta \cdot\epsilon$, where $\epsilon \sim \mathcal{N}(0,\operatorname{Id})$ is a realization of a normalized Gaussian random process (cf. \cite{FranklinNoise}). Generalizing this perspective would lead to $\delta^2 = \|\operatorname{Cov}_\mu\|$
where $\operatorname{Cov}_{\mu}: Y \rightarrow Y$ is the covariance operator of the noise distribution, i.e.,
\begin{align*}
    \langle \operatorname{Cov}_{\mu} v, w \rangle = \E_{\epsilon \sim \mu}\left[\langle \epsilon, v \rangle \langle \epsilon, w \rangle\right]
\end{align*}
for all $v,w \in Y$.
In this chapter, we focus on the spectral architecture, which projects the noise to the singular functions of the operator. Therefore, we employ the following weaker definition of the noise level, 
\begin{align}\label{eq:probSpecNoise}
    \boldsymbol{\delta}(\mu) = \sqrt{ \sup_{n \in \N} \Delta_n(\mu)},
\end{align}
where we define $\Delta_n(\mu)$ as in \eqref{eq:pidelta}. 
In the special case that $\operatorname{Cov}_\mu$ is diagonal with respect to the singular functions $\left\{ v_n\right\}_{n \in \N}$, e.g., for white noise, we have $\|\operatorname{Cov}_\mu\| = \sup_{n \in \N} \Delta_n(\mu)$.
We further note that with this definition we still get $\boldsymbol{\delta}(\mu) \rightarrow 0$ as $ \|\operatorname{Cov}_\mu\| \rightarrow 0$, while the opposite direction is not true in general.

\subsection{Learned regularization methods}

While classical methods control the influence of the regularization functional by the regularization parameter, we predominantly concentrate on the case where the optimal relation between the data fidelity term and the regularization term is learned from the data. In \cite{benning2018modern} a definition of a regularization method has been given, which captures more general aspects appearing in data-driven regularizations, but is still not completely suited for covering all aspects we are interested in (and which are appearing in different learning-based approaches). To clarify our understanding of regularization in this sense, we first restate the definitions of a convergent regularization given in \cite{benning2018modern}. These work for general nonlinear inverse problems with a forward operator $A$ between metric spaces $X$ and $Y$. Given some distance measure $F$ on $Y$ and some possible multivalued selection operator $S$ on $X$, we define a solution of the inverse problem $A(x) =y$ as any
$$ x^* \in S\left(\text{arg}\min_{x \in X} F(A(x),y) \right).$$
Based on these considerations, we define the set-valued solution operator \begin{align*}\mathcal{S}: y \mapsto S\left(\text{arg}\min_{x \in X} F(A(x), y) \right),\end{align*}
which can be seen as a non-linear generalization of the pseudo-inverse.
Accordingly, a set-valued notion of stability is required, which is achieved with the help of Kuratowski convergence. The Kuratowski limit inferior of a sequence of sets $S_n \subset X$ on the metric space $(X,d)$ is defined as 
\begin{align*}
    K-\operatorname*{lim\,inf}_{n \rightarrow \infty} S_n = \left\{x \in X \, \Bigl |\, \operatorname*{lim\,sup}_{n \rightarrow \infty} d(x,S_n) = 0\right\},
\end{align*}
where the distance between an element $x \in X$ and a set $S \subset X$ is defined by
\begin{align*}
    d(x,S) = \inf_{y \in S} d(x,y).
\end{align*}
Using this definition, we call a multivalued operator $R: Y \rightrightarrows X$ \textit{stable}, if it holds for any sequence $\{y_n\}_{n \in \N} \subset Y$, $y_n \rightarrow y \in Y$ that
\begin{align}\label{eq:setstable}
    \emptyset \neq  K-\operatorname*{lim\,inf}_{n \rightarrow \infty} R(y_n) \subset R(y).
\end{align}
\begin{remark}
    In the case of a single-valued operator $R: Y \rightarrow X$, the above notion of stability reduces to continuity: For a convergent sequence $y_n \rightarrow y$ of elements $y,y_n \in Y$ and any $x \in  K-\operatorname*{lim\,inf}_{n \rightarrow \infty} \{R(y_n)\}$ we see that 
    \begin{align*}
       0 = \operatorname*{lim\,sup}_{n \rightarrow \infty} d(x,R(y_n)) \geq \operatorname*{lim\,inf}_{n \rightarrow \infty} d(x,R(y_n)) \geq 0
    \end{align*}
    and therefore, $\lim_{n \rightarrow \infty} d(x, R(y_n)) = 0$. Thus, we derive 
    \begin{align*}
        \emptyset \neq K-\operatorname*{lim\,inf}_{n \rightarrow \infty} \{R(y_n)\} \subset \{R(y)\}\Longleftrightarrow \lim_{n \rightarrow \infty} R(y_n) = R(y).
    \end{align*}
\end{remark}
A regularization in the classical sense is a family of (possibly multivalued) stable operators $R_{\alpha}: Y \rightrightarrows X$, parametrized by $\alpha \in \mathcal{A}$, where $\mathcal{A}$ denotes the set of possible parameters. A regularization is called \textit{convergent}, if there exists a parameter-choice rule $\boldsymbol{\alpha}: (0, + \infty) \times Y \rightarrow \mathcal{A}$, such that 
\begin{align*}
\lim_{\delta \rightarrow 0} \sup\left\{ d_H\left(\mathcal{S}(y), R_{\boldsymbol{\alpha}(\delta, y^{\delta})}(y^\delta)\right) \, \Bigl | \, F(y^{\delta}, y) \leq \delta\right\} = 0
\end{align*}
 for every $y \in  \text{domain}(\mathcal{S})$. To measure the distance between the set of solutions and the set of reconstructions in the metric space $(X,d)$, we choose the Hausdorff metric $d_H$.

\noindent Our goal is now to transfer this understanding of convergent regularizations to a data-driven setting, where the user does not choose the parameters directly, but rather the learning method that determines the optimal parameters. We are interested in the following ingredients:
\begin{itemize}

\item An ideal noise distribution $\nu$ from some family $\Phi$. The latter could consist of rescalings of a single  noise distribution if the noise statistics is known precisely (e.g. Gaussians with mean zero and rescaled covariance operator), some exponential family  or a more general family of noise distributions.

\item A noise level $\boldsymbol{\delta}(\nu)$ as discussed in the previous section (or equally well any other suitable definition of a noise level for a given distribution $\nu$).

\item A training noise model $\mu$ for the noise, which is again an element of a family of distributions $\Psi$ on $Y$  (like $\Phi$). Ideally $\mu=\nu$ or at least an empirical distribution sampled from $\nu$, but there are several reasons to incorporate a change in this distribution. As we shall see below, cases like plug-and-play priors even need a distribution $\mu$ defined on a different space than $\nu$.

\item A prior model $\pi$ for the solutions $x$ of the inverse problems, which is an element of a family $\Pi$ of distributions on $X$.

\end{itemize}
Let us mention that the noise level, the training noise model, and the prior model determine the data-driven regularization in an implicit way. The usual unknowns of a learned method like the parameters of a deep network are determined from a well-specified strategy once the above ingredients are specified. E.g., in a supervised learning approach, the regularization operator at some noise level is learned directly from pairs $(x,Ax+\epsilon)$, $x\sim \pi$, $\epsilon \sim \mu$.  Note that in practice, the definition of the regularization operator depending on $(\mu,\pi)$ may be non-unique due to local minima in the training process or stochastic optimization or initialization. These aspects are however beyond the scope of this chapter.
Now we construct a multivalued operator $R$ from $Y$ to the power set of $X$, parametrized by $\mu$, and $\pi$, which we reflect in the notation as $R_{\mu,\pi}$.

\begin{definition} A family of operators $R_{\mu,\pi}$ is called a (data-driven) family of regularization operators if $R_{\mu,\pi}$ is stable in the sense of \eqref{eq:setstable} for each  $\mu \in \Psi$, $\pi \in \Pi$.
\end{definition}

As a next step, we can define a notion of convergence of the regularization method in a statistical sense.
\begin{definition} A family of regularization operators $R_{\mu,\pi}$ is called convergent if
there exists a parameter choice rule $\mu(\delta,\nu)$ and $\pi(\delta,\nu)$ such that 
\begin{align*}  \lim_{\delta \rightarrow 0} \sup \left \{\E_{\epsilon \sim \nu^\delta} \left [ d_H\left(\mathcal{S}(A(x) ),R_{\mu(\delta,\nu^\delta),\pi(\delta,\nu^\delta)} (A(x) +\epsilon) \right)\right] \, \Bigl | \, \boldsymbol{\delta}(\nu^\delta) \leq \delta \right \} =  0 \end{align*}
for each $x \in X$.
We call the regularization operators convergent over a prior distribution $\pi^*$ if 
\begin{align*} \lim_{\delta \rightarrow 0} \sup \left \{\E_{x \sim \pi^*,\epsilon \sim \nu^\delta} \left [ d_H\left(\mathcal{S}(A(x) ),R_{\mu(\delta,\nu^\delta),\pi(\delta,\nu^\delta)} (A(x) +\epsilon) \right)\right] \, \Bigl | \, \boldsymbol{\delta}(\nu^\delta) \leq \delta \right \} =  0.
 \end{align*}
\end{definition}

Let us make some remarks on the above definition: First of all we concentrate in our notation on the dependence on $\delta$ and $\nu^\delta$, since we vary those when we study convergence. Obviously, there are further dependencies we do not highlight in our notation, like the natural dependence of $\pi$ on $\pi^*$.
Secondly, our setup redefines the common parameter-choice rules for regularization methods as depending on $\delta$ and $\nu^\delta$. We obviously define a parameter-choice rule as an {\em a-priori} choice if it only depends on $\delta$ and as {\em a-posteriori} if it also depends on $\nu^\delta$. Note that an error-free parameter choice for $\pi$ might still yield a convergent regularization method, as we shall see below. In particular, in view of convergence over a prior distribution the ideal choice $\pi = \pi^*$ will typically be successful (and in some cases be possible even in practice). However, a fully error-free parameter choice for $\mu$ and $\pi$ can still not result in a convergent regularization method.

Let us specialize the definitions for a regularization method for a linear forward problem $A:X \rightarrow Y$ on Hilbert spaces $X$ and $Y$, where the natural metric is induced by the squared norm (respectively the scalar product) on the Hilbert space $X$. The natural selection operator is the projection $P$ on $N(A)^\perp$, which effectively leads to the 
pseudo-inverse $A^\dagger y$. Then the method is convergent if 
\begin{align*}\sup_{\nu^\delta \in \Phi, \boldsymbol{\delta}(\nu^\delta) \leq \delta}
\E_{\epsilon \sim \nu^{\delta} } \left[ \Vert R_{\mu(\delta,\nu^\delta),\pi(\delta,\nu^\delta)} (Ax+\epsilon) - P x \Vert^2\right]\rightarrow 0, \end{align*}
for each $x \in X$. Note that since this distance is indifferent to the nullspace component of $x$, we can equally well define convergence as
\begin{align}\label{eq:convergence} \sup_{\nu^\delta \in \Phi, \boldsymbol{\delta}(\nu^\delta) \leq \delta}
\E_{\epsilon \sim \nu^{\delta} }  \left[\Vert R_{\mu(\delta,\nu^\delta),\pi(\delta,\nu^\delta)} (Ax+\epsilon) - x \Vert^2 \right ]\rightarrow 0, \end{align}
for all $x \in N(A)^\perp$. 
Convergence over a prior distribution is then characterized as 
\begin{align*}\sup_{\nu^\delta \in \Phi, \boldsymbol{\delta}(\nu^\delta) \leq \delta}
\E_{x \sim \pi^*} \E_{\epsilon \sim \nu^{\delta} } \left [\Vert R_{\mu(\delta,\nu^\delta),\pi(\delta,\nu^\delta)} (Ax+\epsilon) - x \Vert^2 \right] \rightarrow 0. \end{align*}

\section{Supervised Learning of Spectral Regularization}\label{sec:mse}
In this section, we recapitulate the convergence result from \cite{kabri2024convergent} in the sense of equation \eqref{eq:convergence} and extend it by possible a-priori parameter-choice rules $\mu(\delta)$ for different families of noise distributions. 
As shown for example in \cite{chung2011designing}, the choice of $g$ that minimizes the mean squared error, i.e., 
\begin{align*}
    g^{\text{mse}}(\mu,\pi) = \operatorname*{arg\,min}_{g: \N \rightarrow \R} \E_{x \sim \pi,y \sim \mu}\left[\|R(y, g) - x\|^2\right]
\end{align*}
for noise that is uncorrelated to the data is given by the coefficients
\begin{align}\label{eq:optig}
    g^{\text{mse}}_n(\mu, \pi) = \frac{\sigma_n}{\sigma_n^2 + \frac{\Delta_n(\mu)}{\Pi_n(\pi)}},
\end{align}
or, equivalently using formulation \eqref{eq:tikhRegu},
\begin{align*}
    \lambda^{\text{mse}}_n(\mu, \pi) = \frac{\Delta_n(\mu)}{\Pi_n(\pi)}.
\end{align*}
Due to \cref{ass:data}, the optimal choice of  coefficients is unique and well-defined. We further see that in the spectral framework, the optimal coefficients do not depend on $\mu$ and $\pi$ separately, but rather on their ratio in the directions of each singular value. Thus, unless indicated otherwise,  we choose the error-free parameter choice $\pi(\delta, \nu) = \pi^*$ and write $\Pi_n^* = \Pi_n(\pi^*)$. The optimal spectral reconstruction operator with respect to the mean squared error for training noise drawn from $\mu$ is then denoted by $R^{\text{mse}}_{\mu}$, leading to 
\begin{align*}
    R^{\text{mse}}_{\mu} y = \sum_{n \in \N} g^{\text{mse}}_n(\mu) \langle y , v_n\rangle u_n,
\end{align*}
with $g^{\text{mse}}_n(\mu) = g^{\text{mse}}_n(\mu, \pi^*)$.
\begin{remark}
    In case of a data distribution with covariance operator 
    \begin{align*}
        \operatorname{Cov}_{\pi} x= \sum_{n \in \N} \Pi_n \langle x, u_n \rangle u_n
    \end{align*}
    with eigenvalues $\Pi_n >0$ and a noise distribution with covariance operator 
    \begin{align*}
        \operatorname{Cov}_{\mu} y= \sum_{n \in \N} \Delta_n \langle y, v_n \rangle v_n
    \end{align*}
    with eigenvalues $\Delta_n > 0$,
    the optimal linear spectral regularizer coincides with the optimal linear regularizer (cf. \cite{alberti21}).
\end{remark}
The following Lemma shows that the reconstruction operators obtained in the described way are continuous under reasonable conditions on noise and data.
\begin{lemma}[Continuity]\label{lem:cont}
The reconstruction operator $R^{\text{mse}}_{\mu}: Y \rightarrow X$ is continuous if and only if there exists a constant $c > 0$ such that
\begin{align*}
    \Delta_n(\mu) \geq c \,\sigma_n \Pi_n^*
\end{align*}
for almost all $n \in \N$. In particular, this condition is fulfilled if there exists $c > 0$ and $n_0 \in \N$ such that \begin{align*}\Delta_n(\mu) \geq c \ \Pi_n^*\end{align*}
for all $n \geq n_0$.
\end{lemma}
\begin{proof}
    Fix $\mu$ such that $R^{\text{mse}}_{\mu}: Y \rightarrow X$ is continuous. Since $R^{\text{mse}}_{\mu}$ is linear, this is equivalent to
    \begin{align*}
        \|R^{\text{mse}}_{\mu}\| = \sup_{n \in \N} g^{\text{mse}}_n(\mu) < \infty, 
    \end{align*}
    where in this case $\|\cdot\|$ denotes the  operator norm.
    Inserting the definition of $g^{\text{mse}}_n(\mu)$ we see that this is equivalent to 
    \begin{align*}\inf_{n \in \N} \sigma_n + \frac{\Delta_n(\mu)}{\sigma_n\Pi_n^*} > 0.\end{align*}
This holds if and only if there exists $c > 0$ such that for every $n \in \N$
\begin{align*}
    \frac{\Delta_n(\mu)}{\sigma_n\Pi_n^*} \geq c - \sigma_n.
\end{align*}
Since the sequence of $\sigma_n$ converges to zero, this is equivalent to the first inequality. This also yields the second inequality, as $\sigma_1 \geq \sigma_n$ for all $n \in \N$.
\end{proof}
The foundation of our further studies is the convergence result for learned spectral regularizations derived in \cite[Theorem 3]{kabri2024convergent}, which is restated in the following.
\begin{theorem}[Convergence]\label{thm:convergence}
    The family of learned spectral reconstruction operators $\left\{ R^{\text{mse}}_\mu \right\}_{\mu \in \Psi}$ with
    \begin{align*}
        \Psi= \left\{ \mu \ \Bigl |\  R^{\text{mse}}_{\mu}: Y\rightarrow X \text{ is continuous}\right\}
    \end{align*}
    is a convergent data-driven regularization for $A^\dagger$ and in particular it holds for any $x \in N(A)^\perp$ that
    \begin{align*} \sup_{\nu^\delta \in \Psi, \boldsymbol{\delta}(\nu^\delta) \leq \delta}
\E_{\epsilon \sim \nu^\delta }  \left[\Vert R_{\nu^\delta,\pi(\delta,\nu^\delta)} (Ax+\epsilon) - x \Vert^2 \right ]\longrightarrow 0 \end{align*}
as $\delta \rightarrow 0$.
\end{theorem}
\begin{proof}
    To show convergence for fixed $x \in N(A)^\perp$, we first write down the mean squared error for $\delta > 0$ obtained for fixed problem noise $\nu$ with $\boldsymbol{\delta}(\nu) \leq \delta$ and any training noise $\mu \in \Psi$ as 
\begin{equation}\begin{aligned}\label{eq:errdecomp}
\E_{\epsilon \sim \nu} &\left[ \left\| R^{\text{mse}}_{\mu}(Ax+\epsilon) - x\right\|^2\right]  \\ &
= \sum_{n \in \N} (1-\sigma_n g^{\text{mse}}_n(\mu))^2 \langle x, u_n \rangle^2 + g^{\text{mse}}_n(\mu)^2\Delta_n(\nu) = \\
&=\sum_{n \in \N} \frac{\Delta_n(\mu)^2}{\left( \Pi^*_n \sigma_n^2+ \Delta_n(\mu)\right)^2} \langle x, u_n \rangle^2 + \frac{\Pi^*_n\sigma_n^2\Delta_n(\nu)}{\left(\Pi^*_n\sigma_n^2 + \Delta_n(\mu)\right)^2} \Pi^*_n,
\end{aligned}\end{equation}
and see that the error can be decomposed in a part that is caused by the reconstruction error for uncorrupted data (the first summand) and a part that is caused by the noise (the second summand). We see that for all $n \in \N$ 
\begin{align*}
\frac{\Delta_n(\mu)^2}{\left( \Pi^*_n \sigma_n^2+ \Delta_n(\mu)\right)^2} \leq 1 
\end{align*}
which means that we can bound the series over the first summand independently of the choice of $\mu$ since $\{\langle x, u_n \rangle\}_{n \in \N} \in l^2$ for any $x \in N(A)^\perp$. Choosing the training noise $\mu = \nu$, we can further bound the series over the second summand as 
\begin{align}\label{eq:tbound}
    \frac{\Pi^*_n\sigma_n^2\Delta_n(\nu)}{\left(\Pi^*_n\sigma_n^2 + \Delta_n(\nu)\right)^2} \leq \frac{\Pi^*_n\sigma_n^2\Delta_n(\nu)}{2 \Pi^*_n\sigma_n^2\Delta_n(\nu)} = \frac{1}{2}
\end{align}
and $\sum_{n \in \N} \Pi^*_n < \infty$ by Assumption \ref{ass:data}. Therefore, for any given $\varepsilon > 0$ we can choose $N > 0$ independently of $\nu$ such that
\begin{align*}
    \sum_{n \geq N} \frac{\Delta_n(\nu)^2}{\left( \Pi^*_n \sigma_n^2+ \Delta_n(\nu)\right)^2} \langle x, u_n \rangle^2 + \frac{\Pi^*_n\sigma_n^2\Delta_n(\nu)}{\left(\Pi^*_n\sigma_n^2 + \Delta_n(\nu)\right)^2} \Pi^*_n  \leq \frac{\varepsilon}{2}.
\end{align*}
It remains to show that we can control the remaining finite sum by the choice of the noise level $\delta$. Using that $\Delta_n(\nu) \leq \boldsymbol{\delta}(\nu)^2 \leq \delta^2$ by definition, we get that 
\begin{align*}
\sum_{n \leq N} \frac{\Delta_n(\nu)^2}{\left( \Pi^*_n \sigma_n^2+ \Delta_n(\nu)\right)^2} &\langle x, u_n \rangle^2 + \frac{\Pi^*_n\sigma_n^2\Delta_n(\nu)}{\left(\Pi^*_n\sigma_n^2 + \Delta_n(\nu)\right)^2} \Pi^*_n  \leq  \\ & \leq \delta^2\cdot\left(\frac{\|x\|^2}{\min_{n \leq N} \left\{\Pi^*_n \sigma_n^2\right\}} + \frac{ N}{\sigma_N^2}\right) \leq \frac{\varepsilon}{2}
\end{align*}
as long as $\delta$ fulfills
\begin{align*}
    \delta \leq \frac{\sqrt{\varepsilon}}{\sqrt{2 \cdot \left(\frac{\|x\|^2}{\min_{n \leq N} \left\{\Pi^*_n \sigma_n^2\right\}} + \frac{N}{\sigma_N^2}\right)}}.
\end{align*}
Since $N$ can be chosen independently of $\nu$ so can $\delta$ and we derive the uniform convergence result on the whole set $\Psi$.
\end{proof}
Based on the above result, we investigate possible a-priori parameter-choice rules and reveal the desirable effect of training with so-called white noise. 
\begin{theorem}\label{thm:convergprio_mse}[A-priori choice of training noise]
     Let $\mu: (0,\infty) \rightarrow \Psi$ be an a-priori choice of training noise  with $\boldsymbol{\delta}(\mu(\delta))  = \delta$ while
    \begin{align}\label{eq:decay}
    \Delta_n(\mu(\delta)) \geq \delta^2 \ell(n) 
    \end{align}
    for all $n \in \N$ and a uniform lower bound $\ell: \N \rightarrow \R$. Then it holds for every $x \in N(A)^{\perp}$  and families of noise 
    \begin{align*}
    \Phi_c = \left\{ \nu \ \Bigl | \ \Delta_n(\nu) \leq c\cdot \ell(n) \text{ for all }n \in \N\right\},
    \end{align*}
    with $c > 0$ that 
    \begin{align*}
        \sup_{\nu \in \delta \cdot \Phi_c} \left\{ \E_{\epsilon \sim \nu } \left[ {\left\| R^{\text{mse}}_{\mu(\delta)}(Ax+\epsilon) - x \right\|}^2\right] \right\} \longrightarrow 0,
    \end{align*}
    as $\delta \rightarrow 0$. Considering white training noise, i.e., 
    \begin{align*}
        \Delta_n(\mu(\delta)) = \delta^2
    \end{align*}
    for all $n\in \N$ it holds that 
    \begin{align*}
        \sup_{\nu \in \Psi, \boldsymbol{\delta}(\nu)\leq \delta} \left\{ \E_{\epsilon \sim \nu } \left[ {\left\| R^{\text{mse}}_{\mu(\delta)}(Ax+\epsilon) - x \right\|}^2\right] \right\} \longrightarrow 0,
    \end{align*}
    as $\delta \rightarrow 0$.

\end{theorem}
\begin{proof} 
The proof is along the lines of the proof for Theorem \ref{thm:convergence}. The mean squared error for fixed $x \in N(A)^\perp$ can be decomposed for any choice of $\mu$ and $\nu$ as in \eqref{eq:errdecomp}. For an a-priori parameter choice rule $\mu: (0,\infty) \rightarrow \Psi$ that fulfills the decay requirement \eqref{eq:decay} and $\nu \in \delta \cdot \Phi_c$ we can substitute the bound on the second summand \eqref{eq:tbound} by
\begin{align*}
    \frac{\Pi^*_n\sigma_n^2\Delta_n(\nu)}{\left(\Pi^*_n\sigma_n^2 + \Delta_n(\mu(\delta))\right)^2} \leq \frac{\Pi^*_n\sigma_n^2\Delta_n(\nu)}{2 \Pi^*_n\sigma_n^2\Delta_n(\mu(\delta))} \leq \frac{c}{2},
\end{align*}
as by definition
\begin{align*}
    \Delta_n(\nu) \leq \delta^2\cdot c \cdot \ell(n) \leq c \cdot \Delta_n(\mu(\delta)).
\end{align*}
For white training noise and $\nu \in \Psi$ with $\boldsymbol{\delta}(\nu) \leq \delta$ we derive that $\Delta_n(\nu) \leq \delta^2 = \Delta_n(\mu(\delta))$ for all $n \in \N$ and can thus substitute \eqref{eq:tbound} by 
\begin{align*}
    \frac{\Pi^*_n\sigma_n^2\Delta_n(\nu)}{\left(\Pi^*_n\sigma_n^2 + \Delta_n(\mu(\delta))\right)^2} \leq \frac{\Pi^*_n\sigma_n^2\Delta_n(\nu)}{2 \Pi^*_n\sigma_n^2\Delta_n(\mu(\delta))} \leq \frac{1}{2}.
\end{align*}
The remaining part of the proof can be done analogously to the proof of Theorem \ref{thm:convergence}, where in the case $\nu \in \Phi_c$ with $c \geq 1$ the noise level $\delta$ additionally has to be chosen 
proportional to $1/\sqrt{c}$.
\end{proof}
The arguments of the proofs for convergence of the regularization can further be used to show the convergence of the bias $e_0^{\text{mse}}(\mu)$ of $R_\mu^{\text{mse}}$ {to zero} as $\boldsymbol{\delta}(\mu)$ tends to zero, since for any $N \in \N$ it holds that
\begin{align*}
    e_0^{\text{mse}}(\mu) =  \sum_{n\in \N} \frac{1}{ \left(\frac{\sigma_n^2\, \Pi_n^*}{\Delta_n(\mu)} + 1\right)^2}  \Pi^*_n \leq \boldsymbol{\delta}(\mu)^2\frac{N}{\min_{n \leq N}{\sigma_n^2}} + \sum_{n > N} \Pi_n^*.
\end{align*}
In the remaining sections, we transfer three different recent data-driven reconstruction methods to the spectral setting and analyze how they fit into our framework of data-driven regularizations. 
\section{Regularization by Denoising}\label{sec:denoising}
We first explore two forms of regularization by denoising. Their main difference to the previous approach is that they do not use any information on the singular values of the operator during the optimization of the parameters. Therefore, they can both be classified as so-called plug-and-play approaches, where a denoiser is trained independently of the final application and later plugged into the reconstruction operator for a specific problem. The first approach can be seen as a form of post-processing, where a naive reconstruction is mapped to a regularized reconstruction. The second is related to the usual understanding of plug-and-play methods and considers proximal optimization methods, where the proximal map of the regularization functional is substituted by a denoiser.
In both cases, we consider denoisers
\begin{align}\label{eq:denoiser}
    Dx = \operatorname*{arg\,min}_{x' \in X} \|x-x'\|^2 + J_\lambda(x')= \sum_{n \in \N} \frac{1}{1+\lambda_n} \langle x', u_n\rangle u_n,
\end{align}
which will then lead to a reconstruction operator of the form \eqref{eq:tikhRegu}.
Since the denoising problem resembles an inverse problem for $A = \operatorname{Id}$, we know that for noise drawn from a distribution $\mu$ defined on $X$, the denoiser of the form \eqref{eq:denoiser} minimizes
\begin{align*}
    \E_{x\sim \pi^*, \epsilon \sim \mu} \left[ \| x - D(x + \epsilon) \|^2 \right]
\end{align*}
if its coefficients are given by
\begin{align*}
    \lambda_n(\mu) = \frac{\Tilde{\Delta}_n(\mu)}{\Pi_n^*},
\end{align*}
where we define $\Tilde{\Delta}_n(\mu) = \E_{\epsilon \sim \mu}\left[\langle \epsilon, u_n \rangle^2\right]$ in accordance to \eqref{eq:pidelta} but for noise sampled in the space $X$. 
In the following, we denote this denoiser by $D^{\text{mse}}_{\mu}$.
\subsection{Post-processing approach}\label{sec:post}
The idea of the post-processing approach is to obtain a regularization by first attempting a naive reconstruction approach, which is then corrected by applying a denoiser. {This has been proposed e.g., in \cite{FBPConvNet}, where in a tomographic setup the first step was carried out by filtered backprojection. }In our setting, training a denoiser on noisy data with noise model $\mu$ thus yields the reconstruction operator  
\begin{align*}
    R^{\text{post}}_{\mu}(y) = \left(D^{\text{mse}}_{\mu}\circ A^{\dagger}\right)y = \sum_{n \in \N} \frac{\sigma_n}{\sigma_n^2 + \frac{\sigma_n^2\Tilde{\Delta}_n(\mu)}{\Pi_n^*}} \langle y, v_n\rangle u_n.
\end{align*}
We note that an equivalent reconstruction operator can emerge from a pre-processing approach, in which a denoiser is applied to the  measurements before performing a naive reconstruction.
We easily see that $R^{\text{post}}_{\mu} = R^{\text{mse}}_{\nu}$ means that \begin{align}\label{eq:ratioPost} \frac{\Delta_n(\nu)}{\Tilde{\Delta}_n(\mu)} = \sigma_n^2. \end{align}
Therefore, if we want to use the post-processing approach to obtain the optimal mse-reconstruction operator, the training noise for the denoiser has to decrease slower than the underlying problem noise. However, fulfilling this requirement for $\nu$ being white noise would exceed our framework since it leads to $\boldsymbol{\delta}(\mu) = \infty$, where with a slight abuse of notation we define $\boldsymbol{\delta}(\mu) = \sqrt{\sup_{n \in \N} \Tilde{\Delta}_n(\mu)} $, as we deal with noise defined on $X$ now. Another possibility would be to use smoother training data instead, i.e., choose $\Tilde{\Delta}_n(\mu) = \Delta_n(\nu)$ but use training data sampled from $\pi$ with 
\begin{align*}
    \Pi_n(\pi) = \sigma_n^2 \Pi_n^*
\end{align*}
for all $n \in \N$. The analogous condition in the pre-processing setting is fulfilled if the denoiser is trained on measurements $Ax$, with $x \sim \pi^*$. However, both strategies require the knowledge of the singular values $\sigma_n$ in the training process already and thus are not in line with the idea of training a denoiser without specifically knowing the final problem. In the following, we hence investigate if and how the post-processing approach can yield a convergent regularization while still training on the original data $\pi^*$ and noise with a finite noise level. The first observation we make is that a finite noise level during training can only lead to a continuous reconstruction operator for suitable data distributions $\pi^*$.
\begin{lemma}\label{lem:cont_post}
The reconstruction operator $R^{\text{post}}_{\mu}: Y \rightarrow X$ is continuous if and only if there exists a constant $c > 0$ such that
\begin{align*}
    \sigma_n\Tilde{\Delta}_n(\mu) \geq c \,\Pi_n^*
\end{align*}
for almost all $n \in \N$.    
\end{lemma}
\begin{proof}
    First, we notice that $R^{\text{post}}_{\mu}$ is continuous if and only if $R^{\text{mse}}_{\nu}$ with $\nu$ fulfilling \eqref{eq:ratioPost} is continuous.
    By Lemma \ref{lem:cont} we know that the continuity of $R^{\text{mse}}_{\nu}$ is equivalent to the existence of $c > 0$ such that
    \begin{align*}
    \sigma_n^2 \Tilde{\Delta}_n(\mu) = \Delta_n(\nu) \geq c \, \sigma_n \Pi_n^*
    \end{align*}
    for almost all $n \in \N$. The positivity of all $\sigma_n$ allows to write the above inequality as claimed in the Lemma.
\end{proof}
\begin{corollary}
    Let the reconstruction operator $R^{\text{post}}_{\mu}:Y \rightarrow X$ be continuous for a noise distribution $\mu$ with $\boldsymbol{\delta}(\mu) < \infty$. Then the data distribution $\pi^*$ fulfills
    \begin{align*}
        c \, \Pi_n^* \leq \sigma_n
    \end{align*}
    for a constant $c > 0$ and almost all $n \in \N$.
\end{corollary}
\begin{proof}
    By Lemma \ref{lem:cont_post} we know that the continuity of $R^{\text{post}}_{\mu}$ yields the existence of $\Tilde{c} > 0$ such that
    \begin{align*}
    \sigma_n^2 \Tilde{\Delta}_n(\mu) = \Delta_n(\nu) \geq \Tilde{c} \, \sigma_n \Pi_n^*
    \end{align*}
    for almost all $n \in \N$. Since $\Tilde{\Delta}_n(\mu) \leq \boldsymbol{\delta}(\mu)^2$ by definition, the claimed inequality is fulfilled with $c = \frac{\Tilde{c}}{\boldsymbol{\delta}(\mu)^2}$.
    \end{proof}
Although the optimal mse-regularizer cannot always be obtained by the post-processing approach with training noise of a finite noise level, we can still find a-priori rules to choose the training noise that do not depend on the singular values of the forward operator. Just as continuity, this is only possible under additional assumptions on the data.
\begin{theorem}\label{thm:conv_post}
Let the data distribution $\pi^*$ fulfill
\begin{align*}
    d \, \Pi_n^* \leq \sigma_n^2
\end{align*}
for a constant $d > 0$ and all $n \in \N$, further let $\mu: (0,\infty) \rightarrow \Psi$ be an a-priori choice of training noise  with $\boldsymbol{\delta}(\mu(\delta))  = \delta$ while
    \begin{align}\label{eq:decay2}
    \Tilde{\Delta}_n(\mu(\delta)) \geq \delta^2 \ell(n) 
    \end{align}
    for all $n \in \N$ and a uniform lower bound $\ell: \N \rightarrow \R$. Then it holds for every $x \in N(A)^{\perp}$  and families of noise 
    \begin{align*}
    \Phi^2_c = \left\{ \nu \ \Bigl | \ \Delta_n(\nu) \leq c\cdot \ell(n)^2 \text{ for all }n \in \N\right\}
    \end{align*}
    with $c > 0$ that 
    \begin{align*}
        \sup_{\nu \in \delta^2 \cdot \Phi^2_c} \left\{ \E_{\epsilon \sim \nu } \left[ {\left\| R^{\text{post}}_{\mu(\delta)}(Ax+\epsilon) - x \right\|}^2\right] \right\} \longrightarrow 0
    \end{align*}
    as $\delta \rightarrow 0$. Considering white training noise, i.e., 
    \begin{align*}
        \Tilde{\Delta}_n(\mu(\delta)) = \delta^2
    \end{align*}
    for all $n\in \N$ it holds that 
    \begin{align*}
        \sup_{\nu \in \Psi, \boldsymbol{\delta}(\nu)\leq \delta^2} \left\{ \E_{\epsilon \sim \nu } \left[ {\left\| R^{\text{post}}_{\mu(\delta)}(Ax+\epsilon) - x \right\|}^2\right] \right\} \longrightarrow 0
    \end{align*}
    as $\delta \rightarrow 0$.
    
\end{theorem}
\begin{proof}

Again, the proof is along the lines of the proof for Theorem \ref{thm:convergence}. The mean squared error obtained for training noise $\mu$ and $x \in N(A)^\perp$ can be decomposed into
\begin{align*}
\E_{\epsilon \sim \nu} &\left[ \left\| R^{\text{post}}_{\mu}(Ax+\epsilon) -x \right\|^2\right]  = \sum_{n \in \N} (1-\sigma_n g^{\text{post}}_n(\mu))^2 \langle x, u_n \rangle^2 + g^{\text{post}}_n(\mu)^2\Delta_n(\nu) = \\
&=\sum_{n \in \N} \frac{\sigma_n^4\Tilde{\Delta}_n(\mu)^2}{\left( \Pi_n^* \sigma_n^2+ \sigma_n^2 \Tilde{\Delta}_n(\mu)\right)^2} \langle x, u_n \rangle^2 + \frac{\Pi_n^*\sigma_n^2\Delta_n(\nu)}{\left(\Pi_n^*\sigma_n^2 + \sigma_n^2\Tilde{\Delta}_n(\mu)\right)^2} \Pi_n^*.
\end{align*}For an a-priori parameter choice rule $\mu: (0,\infty) \rightarrow \Psi$ that fulfills the decay requirement \eqref{eq:decay2} and $\nu \in \delta \cdot \Phi^2_c$ we can bound the second summand by
\begin{align*}
    \frac{\Pi^*_n\sigma_n^2\Delta_n(\nu)}{\left(\Pi^*_n\sigma_n^2 + \sigma_n^2\Tilde{\Delta}_n(\mu(\delta))\right)^2} \leq \frac{\Pi^*_n\Delta_n(\nu)}{\sigma_n^2\Tilde{\Delta}_n(\mu)^2} \leq \frac{c}{d},
\end{align*}
where in the last step we use the requirement on $\pi^*$ and
\begin{align*}
    \Delta_n(\nu) \leq \delta^4 \cdot c \cdot \ell(n)^2 \leq c \cdot \Tilde{\Delta}_n(\mu(\delta))^2.
\end{align*}
For white training noise and $\nu \in \Psi$ with $\boldsymbol{\delta}(\nu) \leq \delta^2$ we derive that $\Delta_n(\nu) \leq \delta^4= \Tilde{\Delta}_n(\mu(\delta))^2$ for all $n \in \N$ and can thus bound the second summand by 
\begin{align*}
    \frac{\Pi^*_n\sigma_n^2\Delta_n(\nu)}{\left(\Pi^*_n\sigma_n^2 + \sigma_n^2\Tilde{\Delta}_n(\mu(\delta))\right)^2} \leq \frac{\Pi^*_n\Delta_n(\nu)}{\sigma_n^2\Tilde{\Delta}_n(\mu(\delta))^2} \leq \frac{1}{d}.
\end{align*}
The remaining steps can be carried out in a similar way to the proof of Theorem \ref{thm:convergence} where the choice $N$ to control the tail of the series leads to a slightly different bound \begin{align*}
    \delta \leq \frac{\sqrt{\varepsilon}}{\sqrt{2 \cdot \left(\frac{\|x\|^2}{\min_{n \leq N} \left\{\Pi^*_n\right\}} + \frac{N}{\sigma_N^2}\right)}}.
\end{align*}
    \end{proof}
Similar to the supervised approach, the convergence of the bias $e_0^{\text{post}}(\mu)$ follows from the upper bound 
\begin{align*}
    e_0^{\text{post}}(\mu) =  \sum_{n\in \N} \frac{1}{ \left(\frac{\Pi_n^*}{\Delta_n(\mu)} + 1\right)^2}  \Pi^*_n \leq \boldsymbol{\delta}(\mu)^2 \cdot N + \sum_{n > N} \Pi_n^*,
\end{align*}
that holds for any $N \in \N$.
\subsection{Proximal map approach}\label{sec:prox}
The second approach of regularization by denoising is a simplified version of regularization with plug-and-play priors, which was proposed in \cite{venkatakrishnan2013}. It is based on iterative optimization algorithms for variational problems, which make use of the proximal map corresponding to the regularization functional. The proximal map of a suitable functional $J$ is given for $x \in X$ by \begin{align*}
    \operatorname{prox}_J(x) = \operatorname*{arg\, min}_{x' \in X} \frac{1}{2} \|x -x'\|^2 + J(x').
\end{align*} The idea of plug-and-play priors is to substitute the proximal map of the regularization functional by an already trained denoiser, which in our case has the form \eqref{eq:denoiser}. Due to its specific form we see by insertion that the optimal mse-denoiser $D^{\text{mse}}_{\mu}$ is the proximal map corresponding to the regularization functional 
\begin{align*}
    J^{\text{mse}}_{\mu}(x) = \sum_{n \in \N} \lambda^{\text{mse}}_n(\mu)\langle x,u_n \rangle^2 = \sum_{n \in \N} \frac{\Tilde{\Delta}_n(\mu)}{\Pi_n^*}\langle x,u_n \rangle^2.
\end{align*}
Therefore, inserting $J^{\text{mse}}_{\mu}$ into the variational approach \eqref{eq:tikhRegu} leads to the reconstruction operator
\begin{align*}
    R^{\text{prox}}_\mu(y) = \sum_{n \in \N} \frac{\sigma_n}{\sigma_n^2+\frac{\Tilde{\Delta}_n(\mu)}{\Pi_n^*}} \langle y, v_n\rangle u_n.
\end{align*}
We see that the only difference between the reconstruction operator based on the proximal map approach and the reconstruction operator based on minimizing the mean squared error is the domain in which the noise is added. Therefore, the results from Section \ref{sec:mse} transfer to the setting of spectral plug-and-play priors.
\section{Adversarial Regularization}\label{sec:adversarial}
The idea of using adversarial regularization for inverse problems proposed in \cite{lunz18} is to train the regularization functional as a critic that distinguishes naive reconstructions and ground truth data. The approach is unsupervised in the sense that it does not require pairs of ground truths and reconstructions, but rather independent samples of the real data distribution and the distribution of naive reconstructions. In its original formulation, the regularization functional should be a universal approximator to approximate the $1$-Wasserstein distance between the distribution of the ground truth data $\pi^*$ and the naive reconstructions $\Tilde{\pi}$ via Kolmogorov's duality formula,
\begin{align*}
    W_1(\pi^*, \Tilde{\pi}) = -\inf_{f \in 1\text{-Lip}} \E_{x \sim \pi^*}[f(x)] - \E_{x \sim \Tilde{\pi}}[f(x)],
\end{align*}
where we denote the set of $1$-Lipschitz functions by $1$-Lip. Using the pseudo-inverse $A^\dagger$ as a naive reconstruction operator and noise drawn from $\mu$, we can rewrite the formula as
\begin{align*}
    W_1(\pi^*, \Tilde{\pi}) = -\inf_{f \in 1\text{-Lip}} \E_{x \sim \pi^*}\left[f(x)\right] - \E_{x \sim \pi^*, \epsilon \sim \mu}\left[f\left(A^\dagger(Ax+\epsilon)\right)\right].
\end{align*}
Due to the instability of $A^\dagger$ the above quantity is in general not {finite}. This can be seen {e.g.,} by restricting the set of functions $f$ to spectral regularizers $J_\lambda$ of the form \eqref{eq:Jlambda}. Such a regularization functional parametrized by $\{\lambda_n\}_{n \in \N}$ is $1$-Lipschitz continuous (with respect to the Hilbert space norm) if and only if
$\lambda_n \leq \frac{1}{2}$
for all $n \in \N.$ Therefore we derive the lower bound 
\begin{align*}
    W_1(\pi^*, \Tilde{\pi}) \geq -\inf_{J_{\lambda}\in 1\text{-Lip}} &\E_{x \sim \pi^*}\left[J_{\lambda}(x)\right] - \E_{x \sim \pi^*, \epsilon \sim \mu}\left[J_\lambda\left(A^\dagger(Ax+\epsilon)\right)\right]\\
    &= \sup_{J_{\lambda}\in 1\text{-Lip}}\sum_{n \in \N}  \frac{\lambda_n \Delta_n(\mu)}{\sigma_n^2} = \sum_{n \in \N} \frac{\Delta_n(\mu)}{2\sigma_n^2},
\end{align*}
which means we would have to assume
\begin{align*}
    \sum_{n \in \N} \frac{\Delta_n(\mu)}{\sigma_n^2} < \infty
\end{align*}
for this interpretation. To circumvent this strong restriction, we relax the approach to minimize each summand separately.
\subsection{Relaxed gradient penalty}\label{sec:adv}
Relaxing the requirement of $1$-Lipschitz-continuity, we first consider a parametrization by $\{\lambda^{\text{adv,}\beta}_n\}_{n\in \N}$ which minimizes
\begin{align}\label{eq:wassersteinopti}
    \E_{x}\left[J_{\lambda}(x)\right] - \E_{x, \epsilon \sim \mu}\left[J_{\lambda}(A^{\dagger}(Ax+\epsilon))\right] + \beta\,\E_{x,\epsilon \sim \mu, t \sim U([0,1])}\left[\|\nabla J_\lambda(x_t)\|^2\right],
\end{align}
where $\beta > 0$ controls the penalty term. The gradient is evaluated in convex combinations of ground truth data and reconstructions, with $t$ drawn uniformly from $[0,1]$ and $x_t = x + (1-t) A^{\dagger}\epsilon$.
We recall that
\begin{align*}
     \E_{x}\left[J_{\lambda}(x)\right] - \E_{x, \epsilon\sim \mu}\left[J_{\lambda}(A^{\dagger}(Ax+\epsilon))\right]  =  -\sum_{n \in \N} \frac{\lambda_n \Delta_n(\mu)}{\sigma_n^2}.
\end{align*}
Since further 
\begin{align*}
    \|\nabla J_\lambda(x_t)\|^2 &= \|2 \sum_{n \in \N} \lambda_n \langle x_t, u_n \rangle u_n \|^2 = 4 \sum_{n \in \N} \lambda_n^2 \langle x_t, u_n \rangle^2 \\
    &= 4 \sum_{n \in \N} \lambda_n^2 \left(\langle x, u_n \rangle + \frac{1-t}{\sigma_n
    }\langle \epsilon, v_n\rangle\right)^2 
\end{align*}
and $x, \epsilon$ and $t$ are uncorrelated, we derive that
\begin{align*}
    \beta \, \E_{x,\epsilon, t}\left[\|\nabla J_\lambda(x_t)\|^2\right] &= 4\beta \sum_{n\in \N} \lambda_n^2 \left( \Pi^*_n + \frac{\E_t\left[(1-t)^2\right]}{\sigma_n^2}\Delta_n(\mu) \right) \\ &= 4\beta \sum_{n\in \N} \lambda_n^2 \left( \Pi^*_n + \frac{\Delta_n(\mu)}{3\sigma_n^2} \right). 
\end{align*}
Therefore, we obtain the reconstruction parameters by minimizing
\begin{align*}
    - \frac{\lambda_n \Delta_n(\mu)}{\sigma_n^2} + 4 \beta  \lambda_n^2 \left( \Pi^*_n + \frac{\Delta_n(\mu)}{3\sigma_n^2} \right)
\end{align*}
for each $n \in \N$ separately.
From the first-order optimality conditions, it follows that the optimal choice $\lambda^{\text{adv,}\beta}_n$ fulfills
\begin{align*}
    \lambda^{\text{adv,}\beta}_n(\mu)= \frac{1}{8\beta} \frac{\Delta_n(\mu)}{\sigma_n^2 \left(\Pi^*_n + \frac{\Delta_n(\mu)}{3\sigma_n^2}\right)} =  \frac{3}{8\beta} \frac{\Delta_n(\mu)}{{3} \sigma_n^2 \Pi^*_n + \Delta_n(\mu)}.
\end{align*}
The reconstruction operator obtained by solving \eqref{eq:tikhRegu}
is thus given by 
\begin{align*}
    R^{\text{adv,} \beta}_\mu y= \sum_{n \in \N} \frac{\sigma_n}{\sigma_n^2 + \frac{3}{8\beta}\frac{\Delta_n(\mu)}{(3\sigma_n^2\Pi^*_n + \Delta_n(\mu))}} \langle y, v_n\rangle u_n.
\end{align*}
The additional $\sigma_n^2$ in the denominator of $\lambda^{\text{adv,}\beta}_n(\mu)$ makes it easier to obtain continuous reconstruction operators, as the following Lemma shows.
\begin{lemma}\label{lem:cont_adv}
    The reconstruction operator $R^{\text{adv,} \beta}_{\mu}: Y \rightarrow X$ for $\mu$ with $\boldsymbol{\delta}(\mu) < \infty$ is continuous if and only if there exists a constant $c > 0$ such that
\begin{align*}
    \Delta_n(\mu) \geq c \,\sigma_n^3 \Pi^*_n
\end{align*}
for almost all $n \in \N$. 
\end{lemma}
\begin{proof} 
Analogously to the proof of Lemma \ref{lem:cont} we can show that $R^{adv}_{\mu}: Y \rightarrow X$ is continuous if and only if there exists $c$ such that
\begin{align*}
    \frac{\Delta_n(\mu)}{3\sigma_n^3\Pi^*_n + \sigma_n\Delta_n(\mu)} \geq c
\end{align*}
for almost all $n \in \N$. The above inequality is fulfilled if and only if
\begin{align*}
    \Delta_n(\mu)\geq 3 c\, \sigma_n^3\Pi^*_n + c\, \sigma_n\Delta_n(\mu) 
\end{align*}
for almost all $n \in \N$. Since $\sigma_n \rightarrow 0$ and $\Delta_n(\mu) \leq \boldsymbol{\delta}(\mu)^2 < \infty$, this yields the claimed inequality.
\end{proof}
Studying the behavior of the parameters $\lambda_n^{\text{adv,}\beta}(\mu)$ as $n \rightarrow \infty$, we see that they are always bounded from above by $1$ independently of how the ratio $\Delta_n(\mu)/\Pi_n^*$ evolves. This leads to an additional requirement on the operator, or, respectively, its singular values, to obtain a finite expected reconstruction error for any noise model $\nu$, as for the white noise with $ \Delta_n(\nu) = \boldsymbol{\delta}(\nu)^2$ for all $n \in \N$ the reconstruction of the noise will behave like
\begin{align*}
    \E_{\epsilon \sim \nu} \left[ \left\| R^{\text{adv}}_{\mu}(
    \epsilon)    \right\|^2\right]   &= \sum_{n \in \N}  g^{\text{adv,}\beta}_n(\mu)^2\Delta_n(\nu) = \sum_{n \in \N}  \left(\frac{\sigma_n}{\sigma_n^2 + \frac{3}{8\beta}\frac{\Delta_n(\mu)}{(3\sigma_n^2\Pi^*_n + \Delta_n(\mu))}}\right)^2 \boldsymbol{\delta}(\nu)^2 \\
&\geq \sum_{n\in \N} \left(\frac{\sigma_n}{\sigma_n^2 + \frac{3}{8\beta}}\right)^2  \boldsymbol{\delta}(\nu)^2.
\end{align*}
Since $\sigma_n \rightarrow 0$ the above expression is finite if and only if $\sum_{n \in \N} \sigma_n^2 < \infty$.
 In cases in which the operator fulfills this requirement, the convergence result is comparable to the one for the post-processing approach.
\begin{theorem}\label{thm:convadv}
    Let the singular values of the forward operator $A$ fulfill
\begin{align*}
    \sum_{n\in\N} \sigma_n^2 < \infty,
\end{align*}
 further let $\mu: (0,\infty) \rightarrow \Psi$ be an a-priori choice of training noise  with $\boldsymbol{\delta}(\mu(\delta))  = \delta$ while
    \begin{align}\label{eq:decay3}
    \Delta_n(\mu(\delta)) \geq \delta^2 \ell(n) 
    \end{align}
    for all $n \in \N$ and a uniform lower bound $\ell: \N \rightarrow \R$. Then it holds for every $x \in N(A)^{\perp}$  and families of noise 
    \begin{align*}
    \Phi^2_c = \left\{ \nu \ \Bigl | \ \Delta_n(\nu) \leq c\cdot \ell(n)^2 \text{ for all }n \in \N\right\}
    \end{align*}
    with $c > 0$, that 
    \begin{align*}
        \sup_{\nu \in \delta^2 \cdot \Phi^2_c} \left\{ \E_{\epsilon \sim \nu } \left[ \left\| R^{\text{adv,}\beta}_{\mu(\delta)}(Ax+\epsilon) - x \right\|^2\right] \right\} \longrightarrow 0
    \end{align*}
    for any choice of $\beta > 0$ as $\delta \rightarrow 0$. Considering white training noise, i.e., 
    \begin{align*}
        \Delta_n(\mu(\delta)) = \delta^2
    \end{align*}
    for all $n\in \N$ it holds that 
    \begin{align*}
        \sup_{\nu \in \Psi, \boldsymbol{\delta}(\nu)\leq \delta^2} \left\{ \E_{\epsilon \sim \nu } \left[ \left\| R^{\text{adv},\beta}_{\mu(\delta)}(Ax+\epsilon) - x \right\|^2\right] \right\} \longrightarrow 0
    \end{align*}
    as $\delta \rightarrow 0$.
\end{theorem}
\begin{proof}
     Again we decompose the error obtained for fixed $x \in N(A)^\perp$ into
\begin{align*}
&\E_{\epsilon \sim \nu} \left[ \left\| R^{\text{adv,}\beta}_{\mu}(Ax+\epsilon) - x\right\|^2\right]  \\ &
= \sum_{n \in \N} (1-\sigma_n g^{\text{adv,}\beta}_n(\mu))^2 \langle x, u_n \rangle^2 + g^{\text{adv,} \beta}_n(\mu)^2\Delta_n(\nu) = \\
&=\sum_{n \in \N} \left( \frac{\frac{\Delta_n(\mu)}{3\sigma_n^2\Pi^*_n + \Delta_n(\mu)}}{\sigma_n^2 + \frac{3}{8\beta}\frac{\Delta_n(\mu)}{3\sigma_n^2\Pi^*_n + \Delta_n(\mu)}}\right)^2 \langle x, u_n \rangle^2 + \left(\frac{\sigma_n}{\sigma_n^2 + \frac{3}{8\beta}\frac{\Delta_n(\mu)}{3\sigma_n^2\Pi^*_n + \Delta_n(\mu)}}\right)^2 \Delta_n(\nu).
\end{align*}
For all $n \in \N$ we see that
\begin{align*}
\left( \frac{\frac{\Delta_n(\mu)}{3\sigma_n^2\Pi^*_n + \Delta_n(\mu)}}{\sigma_n^2 + \frac{3}{8\beta}\frac{\Delta_n(\mu)}{3\sigma_n^2\Pi^*_n + \Delta_n(\mu)}}\right)^2 \leq \frac{16}{9}\beta^2,
\end{align*}
which means that we can bound the series over the first summand independently of the choice of $\mu$. Since we are interested in the limit $\delta \rightarrow 0$, we can choose $\delta < 1$ without loss of generality. For an a-priori parameter choice rule $\mu: (0,\infty) \rightarrow \Psi$ that fulfills the decay requirement \eqref{eq:decay3} and $\nu \in \delta \cdot \Phi^2_c$ we can bound the second summand by
\begin{align*}
    \left(\frac{\sigma_n}{\sigma_n^2 + \frac{\Delta_n(\mu(\delta))}{(3\sigma_n^2\Pi^*_n + \Delta_n(\mu(\delta))}}\right)^2 \Delta_n(\nu) &\leq \frac{\Delta_n(\nu)}{\Delta_n(\mu(\delta))^2} \cdot \left (3\sigma_n^2\Pi^*_n+\Delta_n(\mu(\delta))\right)^2 \cdot \sigma_n^2\\
    &\leq c \cdot \max_{n \in \N}\left\{3\sigma_n^2\Pi_n^*+1\right\}^2 \cdot\sigma_n^2,
\end{align*}
where in the last step we use that
\begin{align*}
    \Delta_n(\nu) \leq \delta^4 \cdot c \cdot \ell(n)^2 \leq c \cdot \Delta_n(\mu(\delta))^2.
\end{align*} In the case of white noise and $\boldsymbol{\delta}(\nu) \leq \delta^2$ we get the same estimate with $c=1$. Since both $\sigma_n$ and $\Pi_n^*$ are bounded from above and due to the assumption on $\sigma_n^2$, for any given $\varepsilon > 0$ we can choose $N > 0$ independently of $\nu$ such that
\begin{align*}
     \sum_{n > \N} (1-\sigma_n g^{\text{adv,}\beta}_n(\mu))^2 \langle x, u_n \rangle^2 + g^{\text{adv,} \beta}_n(\mu)^2\Delta_n(\nu) \leq\frac{\varepsilon}{2}.
\end{align*}
It remains to show that we can control the remaining finite sum by the choice of the noise level $\delta$. Using that $\Delta_n(\nu) \leq \boldsymbol{\delta}(\nu)^2 \leq c \cdot \delta^2$, we get that 
\begin{align*}
\sum_{n \leq \N} \left( \frac{\frac{\Delta_n(\mu)}{3\sigma_n^2\Pi^*_n + \Delta_n(\mu)}}{\sigma_n^2 + \frac{3}{8\beta}\frac{\Delta_n(\mu)}{3\sigma_n^2\Pi^*_n + \Delta_n(\mu)}}\right)^2 \langle x, u_n \rangle^2 + \left(\frac{\sigma_n}{\sigma_n^2 + \frac{3}{8\beta}\frac{\Delta_n(\mu)}{3\sigma_n^2\Pi^*_n + \Delta_n(\mu)}}\right)^2 \Delta_n(\nu) \\
\leq \delta^2 \left( \frac{\|x\|^2}{3\,\min_{n \leq N}\{\sigma_n^3\Pi_n^*\} } + \frac{c\cdot N}{\sigma_N^2}\right)\leq \frac{\varepsilon}{2}
\end{align*}
for a suitable choice of $\delta$ which does not depend on $\nu$ and thus yields the uniform convergence result.
\end{proof}
In this setting, the convergence of the bias $e_0^{\text{adv, }\beta}(\mu)$ follows from the upper bound
\begin{align*}
    e_0^{\text{adv, }\beta}(\mu) &=  \sum_{n\in \N} \frac{1}{ \left(\frac{8\beta}{3}\frac{3\sigma_n^4\Pi_n^*}{\Delta_n(\mu)} + \frac{8\beta}{3}\sigma_n^2 + 1\right)^2} \Pi^*_n \\ &\leq \boldsymbol{\delta}(\mu)^2\frac{N}{ 8\beta \cdot \min_{n \leq N}{\sigma_n^4}} + \sum_{n > N} \Pi_n^*.
\end{align*}
\subsection{Promoting a source condition}\label{sec:sc}
In \cite{mukherjee2021learning} it is proposed to promote reconstruction operators that fulfill a source condition. For a solution $x \in X$, the source condition of a variational regularization with a regularization functional $J_\lambda$ is fulfilled if there exists $w \in Y$ such that
\begin{align*}
    A^*w = \nabla J_{\lambda}(x).
\end{align*} Bounding $\|w\|$ makes it possible to obtain convergence rates of classical regularizers (cf. \cite{engl1996regularization}). Generalizing the source condition to the probabilistic context, we thus want to promote $J_{\lambda}$ such that
\begin{align*}
    \E_{x}\left[\|(A^*)^\dagger\nabla J_\lambda(x)\|^2\right]
\end{align*}
is small. Therefore, we now consider a parametrization by $\{\Bar{\lambda}_n\}_{n\in \N}$ which minimizes
\begin{align}\label{eq:sourceopti}
    \E_{x}\left[J_{\lambda}(x)\right] - \E_{x, \epsilon \sim \mu}\left[J_{\lambda}(A^{\dagger}(Ax+\epsilon))\right] + \beta\, \E_{x}\left[\|(A^*)^{\dagger}\nabla J_\lambda(x)\|^2\right]
    ,
\end{align}
where again, $\beta > 0$ controls the penalty term.
We first note that
\begin{align*}
    (A^*)^{\dagger}\nabla J_\lambda(x) = 2 \sum_{n\in \N} \frac{\lambda_n}{\sigma_n} \langle x,u_n \rangle v_n
\end{align*}
and therefore we can write 
\begin{align*}
    \E_{x}\left[\|(A^*)^{\dagger}\nabla J_\lambda(x)\|^2\right] = 4 \, \E_{x}\left[\sum_{n\in \N} \frac{\lambda^2_n}{\sigma^2_n} \langle x,u_n\rangle^2 \right] = 4 \sum_{n\in \N} \frac{\lambda^2_n}{\sigma^2_n} \Pi^*_n.
\end{align*}
Thus, we minimize
\begin{align*}
    -\frac{\lambda_n \Delta_n(\mu)}{\sigma^2_n} + 4 \beta \frac{\lambda_n^2 \Pi^*_n }{\sigma_n^2}
\end{align*}
for each $n$ separately. The first order optimality conditions then lead to\begin{align*}
    \bar{\lambda}_n = \frac{1}{8\beta} \frac{\Delta_n(\mu)}{\Pi^*_n}.
\end{align*}
The reconstruction operator obtained by solving \eqref{eq:tikhRegu}
is thus given by 
\begin{align*}
    R_{\mu}^{\text{sc}, \beta}y = \sum_{n \in \N} \frac{\sigma_n}{\sigma^2_n + \frac{1}{8\beta} \frac{\Delta_n(\mu)}{\Pi^*_n}} \langle y, v_n\rangle u_n.
\end{align*}
Choosing $\beta = 1/8$ we obtain the optimal mse-regularizer $R^{\text{mse}}_\mu$ and therefore, the results from Section \ref{sec:mse} transfer to adversarial reconstruction operators that promote a source condition. Other values of $\beta$ can be treated in a completely analogous way.

\section{Comparison of all approaches}\label{sec:comp}

\setlength{\tabcolsep}{5pt}
\renewcommand{\arraystretch}{1.75}\begin{table}[]
    \centering
    \begin{tabular}{c| c|c|c}
         Optimal parameters & Continuity & Convergence &Description  \\
         \hline
           $\lambda_n^{\text{mse}}(\mu) =\frac{\Delta_n(\mu)}{\Pi^*_n}$ & \multirow{2}{*}{$\Delta_n(\mu) \geq c \,\sigma_n\Pi^*_n$ }& \multirow{2}{*}{$\Delta_n(\mu) \geq c\,\Delta_n(\nu)$} & \cref{sec:mse}\\
         \cline{1-1} \cline{4-4}
           $\lambda_n^{\text{sc}, \beta}(\mu) = \frac{1}{8\beta}\frac{\Delta_n(\mu)}{\Pi^*_n}$ & & & \cref{sec:sc} \\
         \hline
         $\lambda_n^{\text{prox}}(\mu) =
         \frac{\Tilde{\Delta}_n(\mu)}{\Pi^*_n}$ & $\Tilde{\Delta}_n(\mu) \geq c \,\sigma_n\Pi^*_n$ & $\Tilde{\Delta}_n(\mu) \geq c\,\Delta_n(\nu)$ &\cref{sec:prox}\\
         \hline
           \multirow{3}{*}{$\lambda_n^{\text{post}}(\mu) =\frac{  \sigma_n^2\Tilde{\Delta}_n(\mu)}{\Pi^*_n}$} & \multirow{3}{*}{$\sigma_n\,\Tilde{\Delta}_n(\mu) \geq c \,\Pi^*_n$} &$\sigma^2_n\geq c \,\Pi^*_n$
           \\
         && and & \cref{sec:post}\\
         && $\Tilde{\Delta}_n(\mu)^2 \geq c\,\Delta_n(\nu)$ & \\
         \hline
         \multirow{3}{*}{$\lambda_n^{\text{adv,}\beta}(\mu) = \frac{3}{8\beta}\frac{\Delta_n(\mu)}{(3\sigma_n^2\Pi^*_n + \Delta_n(\mu))}$}& \multirow{3}{*}{$\Delta_n(\mu) \geq c \,\sigma_n^3\Pi^*_n$ }&$\sum_{n \in \N} \sigma^2_n < \infty$&\\
          & & and & \cref{sec:adv}\\
         && $\Delta_n(\mu)^2 \geq c\,\Delta_n(\nu)$&\\
    \end{tabular}
    \caption{Comparison of the main results for the discussed approaches.}
    \label{tab:comparison}
\end{table}
In this section, we compare all of the approaches that were considered before. \Cref{tab:comparison} sums up the main results of each section. We visualize our findings by numerical toy examples for Computed Tomography implemented in Python\footnote{Jupyter Notebooks to recreate our examples can be found here: \url{https://github.com/samirak98/learnedSpectral}.}. Therefore, we use the package \texttt{radon.py}\footnote{The package \texttt{radon.py} can be found here:  \url{https://github.com/AlexanderAuras/radon}.} which we originally used in \cite{kabri2024convergent} as it allows for an explicit computation of the forward operator matrix and thus, its singular value decomposition. As ground truth images, we use a subset of $2400$ instances of the ground truth test data from the LoDoPaB-CT Dataset (cf. \cite{Leuschner_2021,leuschner2019_dataset}). In all cases we use Gaussian noise where we choose the covariance matrix to be diagonal with respect to the singular vectors of the forward operator, with the eigenvalues being the desired noise coefficients. We note that in case of white noise, a matrix that is diagonal with respect to the canonical basis is already diagonal in the eigenvector basis.\\Due to the specific architecture of the reconstruction operator, all approaches have in common that the obtained parameters, as well as the requirements for continuity and convergence can be expressed by means of the coefficients $\sigma_n$, $\Pi_n^*$, $\Delta_n(\mu)$ and $\Delta_n(\nu)$. As before, $\mu$ denotes the training noise and $\nu$ denotes the real problem noise. \\The first observation we make is that the optimal reconstruction operator with respect to the mean squared error, $R_{\mu}^{\text{mse}}$, can equivalently be obtained by the adversarial regularization approach that promotes a source condition, $R_{\mu}^{\text{sc,}1/8}$, and the approach using a plug-and-play prior, $R_{\mu}^{\text{prox}}$. In the latter case the noise is added in the space of ground truths, $X$, instead of the 
space of measurements, $Y$. \\
With the post-processing approach, $R^{\text{post}}_\mu$, and the adversarial regularization approach that does not penalize the source condition term, but the spectral norm of the gradient, $R^{\text{adv,}\beta}_{\mu}$, we also investigated a denoising approach and an adversarial approach that cannot retrieve the mse-regularizer in general.
\subsection{Continuity} 

\begin{figure}
    \centering
   \includegraphics[width=0.9\textwidth]{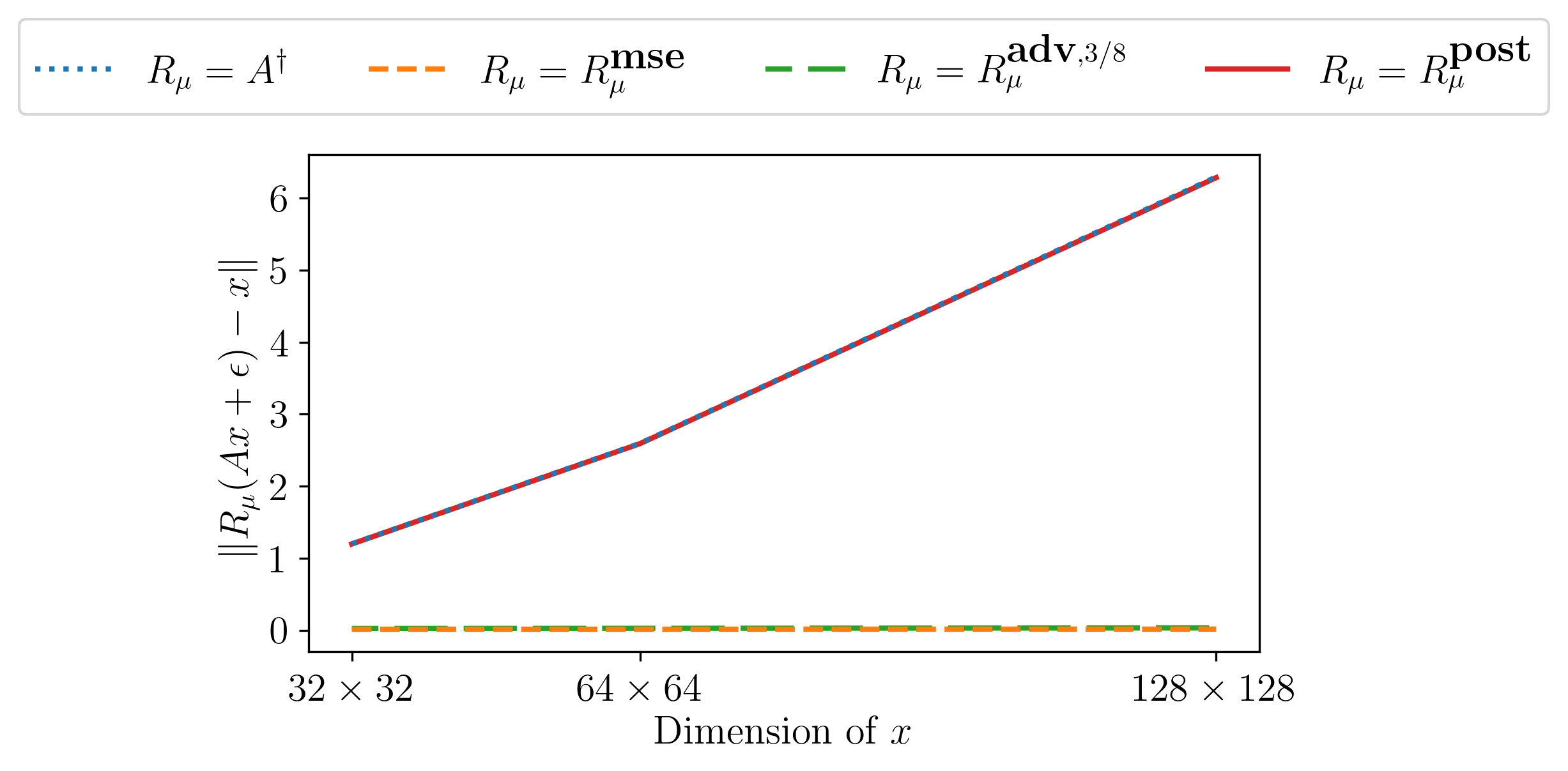}
    \caption{Behavior of reconstruction error for the different approaches applied to different input dimensions. In all cases we choose $\epsilon, \|\epsilon\| = 0.001$ to be a perturbation parallel to $v_{\max}$, which denotes the singular vector that corresponds to the smallest singular value. The training noise which determines the data-driven reconstruction operators is chosen to be white Gaussian noise with $\boldsymbol{\delta}(\mu) = 0.001$.   }
    \label{fig:continuity}
\end{figure}
For all approaches, continuity is equivalent to a lower bound on the decay of the noise coefficients $\Delta_n(\mu)$ by the decay of the data coefficients $\Pi_n^*$, with a factor depending on $\sigma_n$ which differs between the approaches. The weakest condition is obtained for the adversarial approach $R^{\text{adv,}\beta}_{\mu}$, followed by the approaches minimizing the mean squared error $R_{\mu}^{\text{mse}}$, $R_{\mu}^{\text{sc,}\beta}$ and $R_{\mu}^{\text{prox}}$. In all of these cases, it is sufficient if noise and data fulfill the operator-independent requirement
    \begin{align*}
    \Delta_n(\mu) \geq c \,\Pi_n^*\text{,\quad or, \quad} \Tilde{\Delta}_n(\mu) \geq c \,\Pi_n^*\text{, \quad respectively,}
\end{align*}
for almost all $n \in \N$ and a constant $c > 0$, as the singular values $\sigma_n$ accumulate at zero. In case of the post-processing operator $R^{\text{post}}_\mu$ with $\boldsymbol{\delta}(\mu)$, this property is not only not sufficient for continuity, it is further necessary that the data distribution fulfills the operator-dependent requirement
\begin{align*}
    \sigma_n \geq c \, \Pi_n^*
\end{align*}
for almost all $n \in \N$ and a constant $c > 0$. In \cref{fig:continuity}, we visualize the stability or instability of the considered approaches by comparing the reconstruction errors for different input dimensions. In the discrete setting, all of our approaches are mathematically continuous by their linearity. Still, we see that for our specific example, the approaches 
$R_{\mu}^{\text{mse}}$ and $R_{\mu}^{\text{adv,}3/8}$ are clearly more stable than the post-processing approach, $R_{\mu}^{\text{post}}$, which performs as good as naive reconstruction with the pseudo-inverse $A^\dagger$. This observation is in line with the theoretical insight emerging from \cref{lem:cont,lem:cont_post,lem:cont_adv}, stating that the requirements for a $R_{\mu}^{\text{post}}$ to be stable are stronger than for the other approaches we consider (see also \cref{tab:comparison}). Formally, the post-processing approach could be as stable as the other approaches, but in general, that would require exploding training noise.

\subsection{Convergence}

\begin{figure}
    \centering
    \begin{tikzpicture}  
    \node at (0,2) {\small $\Delta_n(\nu) \propto 1$};
    \node at (3.5,2) {\small $\Delta_n(\nu) \propto 1/\sqrt{n}$};
    \node at (7,2) {\small $\Delta_n(\nu) \propto  1/n^4$};
    \node at (0,0)[rotate=90] {
    \includegraphics[scale=0.75]{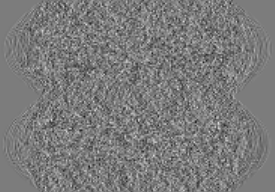}};
    \node at (3.5,0) [rotate=90]{
    \includegraphics[scale=0.75]{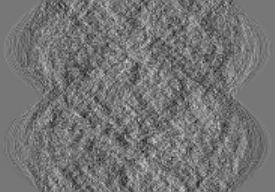}};
    \node at (7,0) [rotate=90]{
    \includegraphics[scale=0.75]{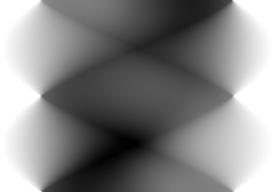}};
    \end{tikzpicture}
    \caption{Examples of noise samples with different decays of $\Delta_n(\nu)$.}
    \label{fig:noise}
\end{figure}
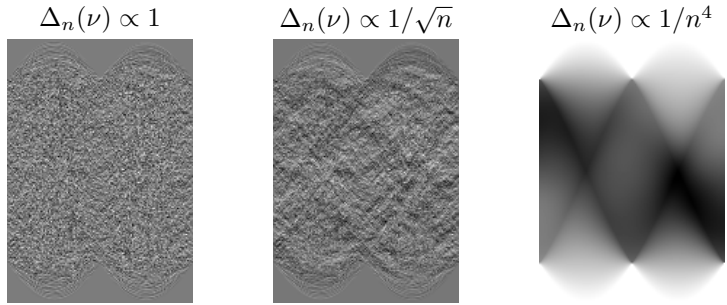
\begin{figure}
    \centering
    \includegraphics[width=\textwidth]{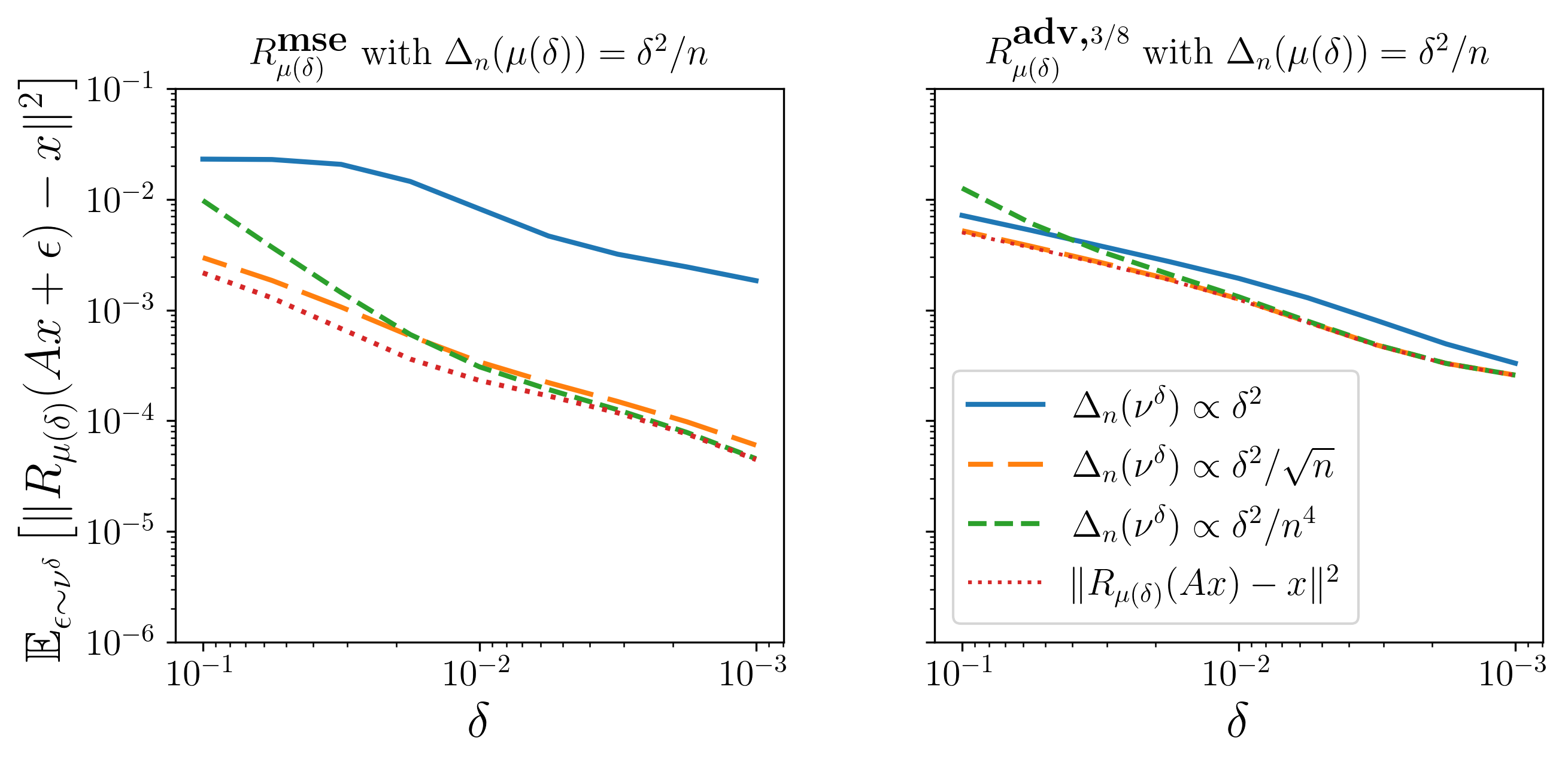}
    \caption{Behavior of the reconstruction error obtained with the supervised and the adversarial approach. We test the performance on measurements corrupted with noise drawn from different kinds of noise models and decaying noise level.}
    \label{fig:convergence_mse_adv}
\end{figure}
\begin{figure}
    \centering
    \includegraphics[width=\textwidth]{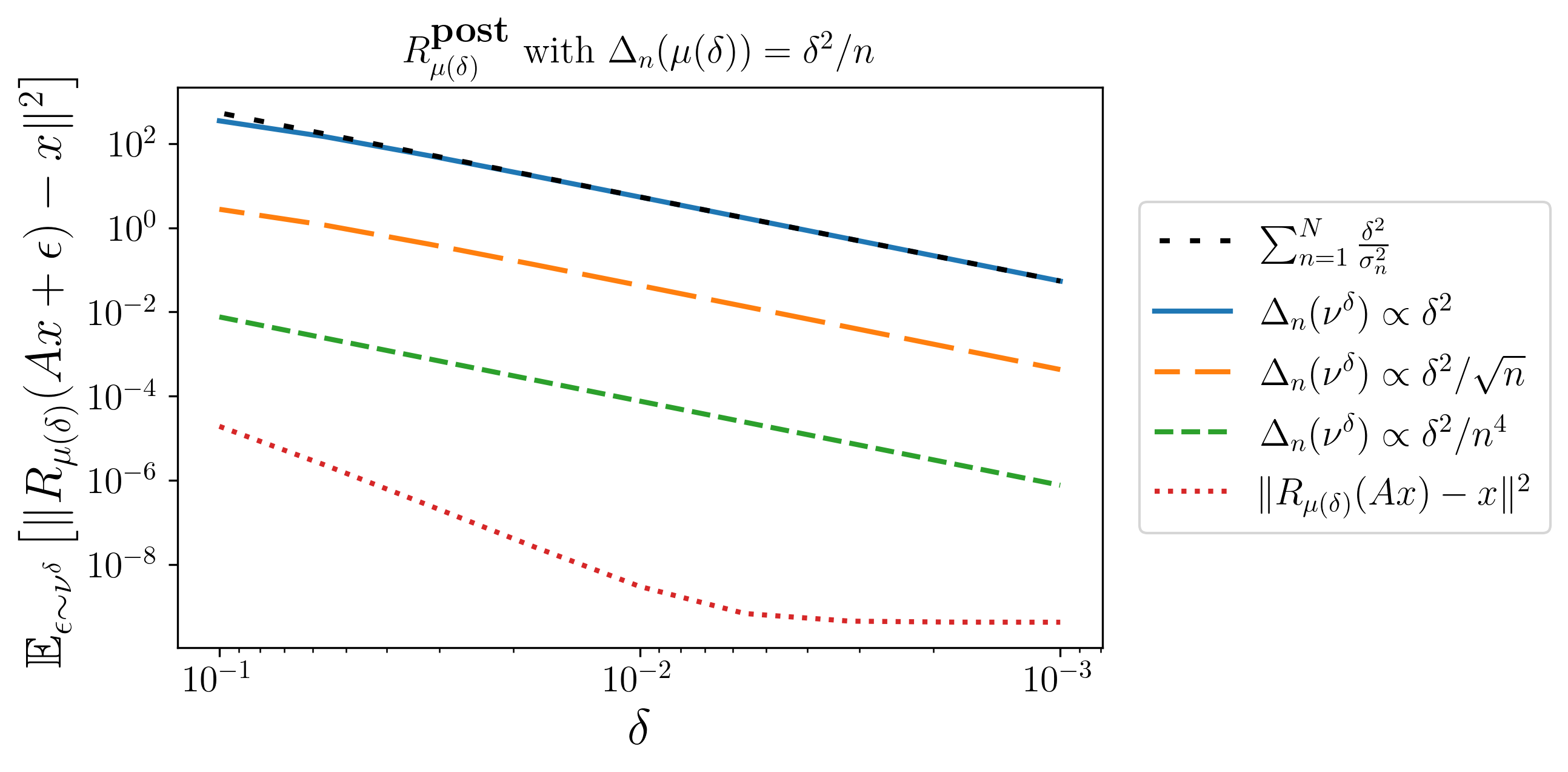}
    \caption{Behavior of the reconstruction error obtained with the post-processing approach. We test the performance on measurements corrupted with noise drawn from different kinds of noise models and decaying noise level.}
    \label{fig:convergence_post}
\end{figure}
In \cref{thm:convergence}, we showed that choosing the training noise $\mu$ to be the real problem noise $\nu$ yields a convergent data-driven regularization method for the supervised and the adversarial approach with source condition. This result easily transfers to the proximal map approach $R_{\mu}^{\text{prox}}$, with the adjustment that instead of $\mu(\nu) = \nu$, we choose $\mu$ such that $\Tilde{\Delta}_n(\mu) = \Delta_n(\nu)$ for all $n \in \N$. For the two other approaches (post-processing and adversarial regularization with gradient penalty), the method obtained by choosing the training noise such that its coefficients coincide with $\Delta_n(\nu)$ is in general not convergent in the sense of \eqref{eq:convergence}. All our convergence results have in common that the noise coefficients of the training noise are required to decay at most as fast as the real noise coefficients. In the case of summable coefficients $\{\Delta_n(\mu)\}_{n\in\N}$ and $\{\Delta_n(\nu)\}_{n\in\N}$ this means that
\begin{align*}
    \E_{\epsilon \sim \nu}\left[\|\epsilon\|^2\right] = \sum_{n \in \N} \Delta_n(\nu) \leq \frac{1}{c}\sum_{n \in \N} \Delta_n(\mu) = \frac{1}{c} \, \E_{\epsilon \sim \mu}\left[\|\epsilon\|^2\right],
\end{align*}
i.e., up to a uniform constant $c$, the variance of the problem noise should be bounded from above by the variance of the training noise. In the general case, we can say that, up to a uniform constant $c$, the problem noise should be weaker than the training noise to obtain a convergent regularization. An illustration of different decays of noise can be found in \cref{fig:noise}. \\
Additional assumptions on the operator and/or the data distribution need to be made for the convergence of $R^{\text{post}}_\mu$ and $R^{\text{adv,}\beta}_{\mu}$. In  \cref{fig:convergence_mse_adv,fig:convergence_post}, we compare the expected reconstruction error that is obtained by reconstructing one noisy data instance with the different approaches. We assume all reconstruction operators to be trained with noise $\mu(\delta)$ that fulfills $\Delta_n(\mu(\delta)) = \delta^2/\sqrt{n}$. We then compute the expected reconstruction error based on test noise drawn from three different noise models and decaying noise level. As it was also shown in the proofs of \cref{thm:convergence,thm:convergprio_mse,thm:conv_post,thm:convadv}, the reconstruction error for each discussed method $R_{\mu}$ can be decomposed into a noise-dependent and a data-dependent term, as
\begin{align*}
     \E_{\epsilon\sim \nu^\delta}\left[\|R_{\mu}(Ax+\epsilon) - x\|^2\right] = \|x_0\|^2 +   \|R_{\mu}(Ax)-x\|^2 +\E_{\epsilon\sim \nu^\delta}\left[\|R_{\mu}\epsilon\|^2\right].
\end{align*}
Therefore, the error is bounded from below by the term 
\begin{align}\label{eq:xbias}
    \|R_{\mu}(Ax)-x\|^2,
\end{align}
which does not depend on the noise. The convergence of \eqref{eq:xbias} to the null-space remainder $\|x_0\|^2$ was shown for all methods without further requirements. Accordingly, the corresponding terms from the numerical examples seem to converge to a value close to zero in all cases. We also see that the data-dependent error decays much faster for the post-processing approach, than for the other approaches. Moreover, most of the error curves for $R^{\text{mse}}_{\mu(\delta)}$ and $R^{\text{adv,} 3/8}_{\mu(\delta)}$ are dominated by \eqref{eq:xbias}, which indicates convergence of the whole error. In general, we see a stronger decay of the error, for a stronger decay of the noise coefficients. In particular, $R^{\text{mse}}_{\mu(\delta)}$ tested on white noise (i.e., $\Delta_n(\nu^\delta) \propto \delta^2$), is clearly dominated by the noise-dependent part. 
This might indicate that the error would not converge to the optimum $\|x_0\|^2$ in a continuum limit of the reconstruction problem. However, we stress that in a finite-dimensional, discretized setting with $N < \infty$ singular values, it holds for the noise-dependent part of the error that
\begin{align}\label{eq:noiseerrorbound}
    \E_{\epsilon\sim \nu^\delta}\left[\|R_{\mu}\epsilon\|^2\right] \leq \sum_{n = 1}^N \frac{\Delta_n(\nu^\delta)}{\sigma_n^2} \leq \boldsymbol{\delta}(\nu^\delta)^2 \cdot c,
\end{align}with $c<\infty$. This automatically yields convergence to zero, as the noise level of the test noise tends to zero. In the continuum limit of an unstable inverse problem, the constant $c$ diverges and thus, the bound becomes trivial. 
This is a useful insight to inspect the numerical convergence rates of $R^{\text{post}}_{\mu(\delta)}$ shown in \cref{fig:convergence_post}: All error curves decay proportional to $\delta^2$, but the decay for testing on white noise coincides with the upper bound \eqref{eq:noiseerrorbound}. This can be a sign for divergence of the error in the infinite dimensional limit. In comparison to the other considered approaches, we observe a slower convergence of the noise-dependent error for all choices of test noise.
The qualitative difference between the obtained reconstructions can be seen in \cref{fig:recos_mse,fig:recos_adv,fig:recos_post}.

\section{Conclusion}
In this chapter we compared the convergence behavior of different learning paradigms for the solution of inverse problems. The definition of convergent data-driven regularizations we gave in \cref{sec:theory} is strongly connected to the classical understanding introduced in \cite{engl1996regularization, benning2018modern}, but takes into account the statistical nature of the problem and practical considerations of users. Using the spectral architecture and a quadratic error functional, we first derived conditions for the convergence of reconstruction operators that are optimized in a supervised approach. Keeping the architecture and the quadratic nature of the error functional, we investigated various other learning strategies: Post-processing, plug-and-play priors by learning a proximal map and adversarial regularizers with and without promoting a source condition. We could show that the optimal mse-regularizer can as well be obtained by all of these approaches, as long as the training noise can be chosen accordingly. Therefore, the main differences of the approaches, that came forward in both theoretical results and numerical examples, can be summarized as follows.

In all cases, \textit{continuity} is equivalent to a decay condition on the ratio of the training noise variance and the data variance. In most cases, the decay has to be bounded from below by the decay of the singular values of the forward operator. The post-processing approach turned out to be the only approach that requires this ratio to decay at least as slow as the \textit{inverse} singular values of the forward operator. Thus, the more stable the forward operator, the harder it is to obtain a continuous reconstruction operator with the post-processing approach. Moreover, in cases where the data variance decays slower than the singular values, a continuous post-processing operator cannot be obtained with training noise of finite noise level. 

\textit{Convergence} of the considered data-driven regularizations is always connected to the decay of the ratio of the training noise variance and the variance of the real noise. For the mse-regularizer, as well as the proximal map approach and the adversarial regularizer with source condition, it is sufficient that the variance of the training noise does not decay faster than the variance of the real noise. For the post-processing approach and the conventional adversarial approach the same has to hold for the \textit{quadratic} variance of the training noise, which is a stronger requirement. Additionally, these two strategies impose a requirement on the singular values of the forward operator. 

In all scenarios, training with white noise yields a convergent data-driven regularization, as long as possible requirements on data and forward operator are fulfilled. To conclude, we note that the general insights about the role of the relationship between noise, data and forward operator we gain from investigating the spectral architecture, are likely to be applicable for more complex architectures and optimization schemes. In particular we expect our statement on the relation between the training noise and the real noise to be a rather universal one.     
\begin{figure}
    \centering
    \begin{tikzpicture}  
    \node at (0,1.5) {\small $\Delta_n(\nu) \propto \delta^2$};
    \node at (2.6,1.5) {\small $\Delta_n(\nu) \propto \delta^2/\sqrt{n}$};
    \node at (5.2,1.5) {\small $\Delta_n(\nu) \propto  \delta^2/n^4$};
    \node at (-1.5,0) [rotate=90] {$\delta=0.1 $};
    \node at (-1.5,-2.6) [rotate=90] {$\delta=0.01 $};
    \node at (-1.5,-5.2) [rotate=90] {$\delta=0.001 $};
    \node at (0,0)[rotate=90] {
    \includegraphics[scale=0.75]{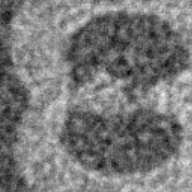}};
    \node at (2.6,0) [rotate=90]{
    \includegraphics[scale=0.75]{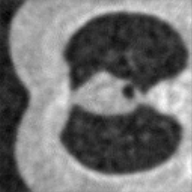}};
    \node at (5.2,0) [rotate=90]{
    \includegraphics[scale=0.75]{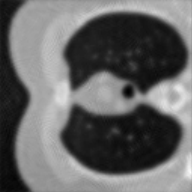}};
    \node at (0,-2.6)[rotate=90] {
    \includegraphics[scale=0.75]{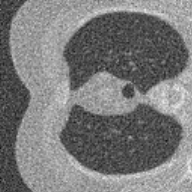}};
    \node at (2.6,-2.6) [rotate=90]{
    \includegraphics[scale=0.75]{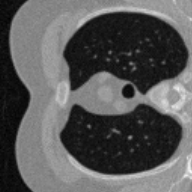}};
    \node at (5.2,-2.6) [rotate=90]{
    \includegraphics[scale=0.75]{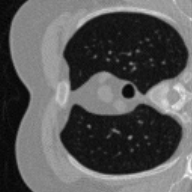}};
    \node at (0,-5.2) [rotate=90]{
    \includegraphics[scale=0.75]{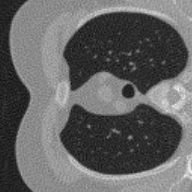}};
    \node at (2.6,-5.2) [rotate=90]{
    \includegraphics[scale=0.75]{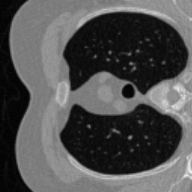}};
    \node at (5.2,-5.2) [rotate=90]{
    \includegraphics[scale=0.75]{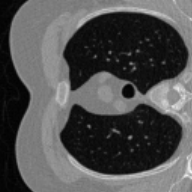}};
    \end{tikzpicture}
    \caption{Reconstructions obtained with $R^{\text{mse}}_{\mu(\delta))}$ where $\Delta_n(\mu(\delta)) = \frac{\delta^2}{\sqrt{n}}$.}
    \label{fig:recos_mse}
\end{figure}
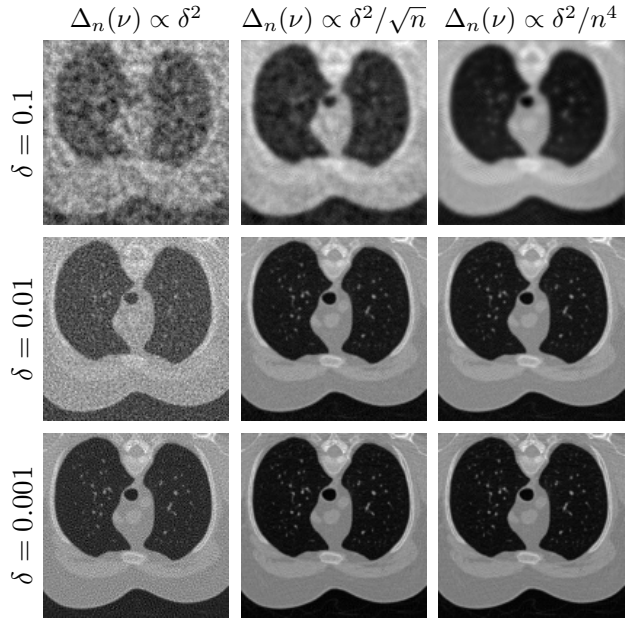
\begin{figure}
    \centering
    \begin{tikzpicture}  
    \node at (0,1.5) {\small $\Delta_n(\nu) \propto \delta^2$};
    \node at (2.6,1.5) {\small $\Delta_n(\nu) \propto \delta^2/\sqrt{n}$};
    \node at (5.2,1.5) {\small $\Delta_n(\nu) \propto  \delta^2/n^4$};
    \node at (-1.5,0) [rotate=90] {$\delta=0.1 $};
    \node at (-1.5,-2.6) [rotate=90] {$\delta=0.01 $};
    \node at (-1.5,-5.2) [rotate=90] {$\delta=0.001 $};
    \node at (0,0)[rotate=90] {
    \includegraphics[scale=0.75]{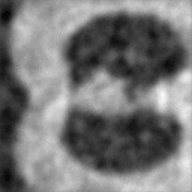}};
    \node at (2.6,0) [rotate=90]{
    \includegraphics[scale=0.75]{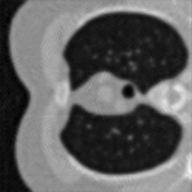}};
    \node at (5.2,0) [rotate=90]{
    \includegraphics[scale=0.75]{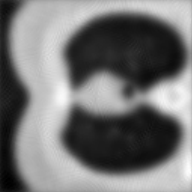}};
    \node at (0,-2.6)[rotate=90] {
    \includegraphics[scale=0.75]{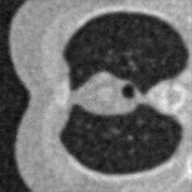}};
    \node at (2.6,-2.6) [rotate=90]{
    \includegraphics[scale=0.75]{images/adv_training0.01.pdf}};
    \node at (5.2,-2.6) [rotate=90]{
    \includegraphics[scale=0.75]{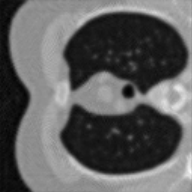}};
    \node at (0,-5.2) [rotate=90]{
    \includegraphics[scale=0.75]{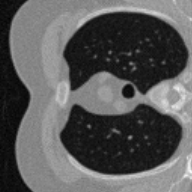}};
    \node at (2.6,-5.2) [rotate=90]{
    \includegraphics[scale=0.75]{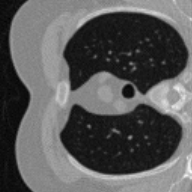}};
    \node at (5.2,-5.2) [rotate=90]{
    \includegraphics[scale=0.75]{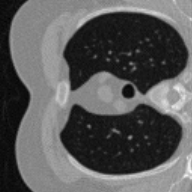}};
    \end{tikzpicture}
    \caption{Reconstructions obtained with $R^{\text{adv}}_{\mu(\delta))}$ where $\Delta_n(\mu(\delta)) = \frac{\delta^2}{\sqrt{n}}$.}
    \label{fig:recos_adv}
\end{figure}
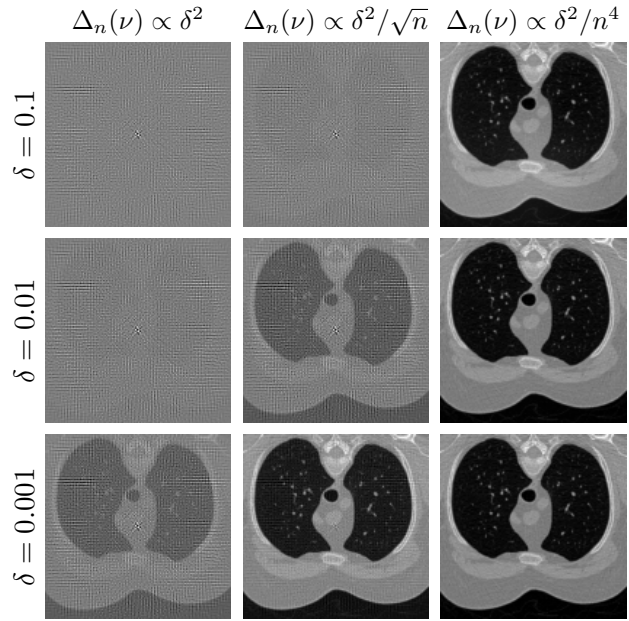
\begin{figure}
    \centering
    \begin{tikzpicture}  
    \node at (0,1.5) {\small $\Delta_n(\nu) \propto \delta^2$};
    \node at (2.6,1.5) {\small $\Delta_n(\nu) \propto \delta^2/\sqrt{n}$};
    \node at (5.2,1.5) {\small $\Delta_n(\nu) \propto  \delta^2/n^4$};
    \node at (-1.5,0) [rotate=90] {$\delta=0.1 $};
    \node at (-1.5,-2.6) [rotate=90] {$\delta=0.01 $};
    \node at (-1.5,-5.2) [rotate=90] {$\delta=0.001 $};
    \node at (0,0)[rotate=90] {
    \includegraphics[scale=0.75]{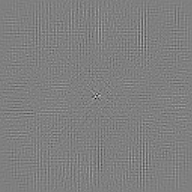}};
    \node at (2.6,0) [rotate=90]{
    \includegraphics[scale=0.75]{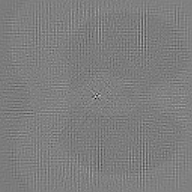}};
    \node at (5.2,0) [rotate=90]{
    \includegraphics[scale=0.75]{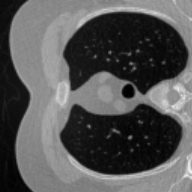}};
    \node at (0,-2.6)[rotate=90] {
    \includegraphics[scale=0.75]{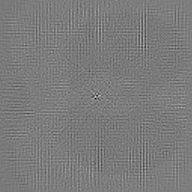}};
    \node at (2.6,-2.6) [rotate=90]{
    \includegraphics[scale=0.75]{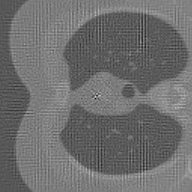}};
    \node at (5.2,-2.6) [rotate=90]{
    \includegraphics[scale=0.75]{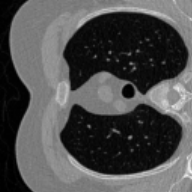}};
    \node at (0,-5.2) [rotate=90]{
    \includegraphics[scale=0.75]{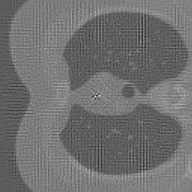}};
    \node at (2.6,-5.2) [rotate=90]{
    \includegraphics[scale=0.75]{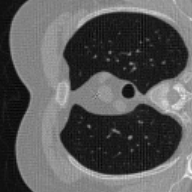}};
    \node at (5.2,-5.2) [rotate=90]{
    \includegraphics[scale=0.75]{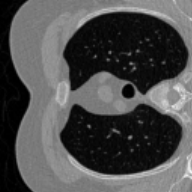}};
    \end{tikzpicture}
    \caption{Reconstructions obtained with $R^{\text{post}}_{\mu(\delta))}$ where $\Delta_n(\mu(\delta)) = \frac{\delta^2}{\sqrt{n}}$.}
    \label{fig:recos_post}
\end{figure}

\section*{Acknowledgement}
This research was supported in part through the Maxwell computational resources operated at Deutsches Elektronen-Synchrotron DESY, Hamburg, Germany. The authors acknowledge support from DESY (Hamburg, Germany),
a member of the Helmholtz Association HGF and from the German Research Foundation, project BU 2327/19-1.

\bibliographystyle{plain}
\bibliography{bibliography}
\end{document}